\documentclass[aap]{imsart}

\RequirePackage{amsthm,amsmath,amsfonts,amssymb}
\RequirePackage[numbers]{natbib}

\startlocaldefs
\theoremstyle{plain}
\newtheorem{assumption}{Assumption}
\newtheorem{theorem}{Theorem}[section]
\newtheorem{lemma}[theorem]{Lemma}
\newtheorem{corollary}[theorem]{Corollary}
\theoremstyle{remark}
\newtheorem{definition}[theorem]{Definition}
\newtheorem{example}{Example}
\newtheorem*{remark}{Remark}
\usepackage{mathtools}
\usepackage{xcolor}
\usepackage{enumitem}

\DeclarePairedDelimiter\abs\lvert\rvert
\DeclarePairedDelimiter\norm\lVert\rVert

\newcommand{\NN}{\mathbb{N}}
\newcommand{\ZZ}{\mathbb{Z}}
\newcommand{\Ii}{\mathbb{I}}
\newcommand{\Zd}{{\ZZ^d}}
\newcommand{\vZ}{{v\in \Zd}}
\newcommand{\R}{\mathbb{R}}
\newcommand{\Rd}{{\R^d}}
\newcommand{\PP}{\mathbb{P}}
\newcommand{\EE}{\mathbb{E}}
\newcommand{\dd}{\mathrm{d}}
\newcommand{\Jz}{{J_z^{n}}}
\newcommand{\Hz}{{H_z^{n}}}
\newcommand{\dnm}{D_{n}^{-}}
\newcommand{\dnp}{D_{n}^{+}}
\newcommand{\Bb}{\mathcal{B}}
\newcommand{\I}[1]{1_{#1}}
\newcommand{\Pn}{{P_{n}}}
\newcommand{\Qn}{{Q_{n}}}
\newcommand{\Av}{{A_v^{n}}}
\newcommand{\Avm}{{A_v^{(m)}}}
\newcommand{\An}{{A_0^{n}}}
\newcommand{\Anm}{{A_0^{(m)}}}
\newcommand{\Bvm}{{B_v^{(m)}}}
\newcommand{\Bm}{{B^{(m)}}}
\newcommand{\bb}[1]{\boldsymbol{#1}}
\newcommand{\ii}{\mathrm{int}\,}
\newcommand{\cl}{\mathrm{cl}\,}
\newcommand{\vd}{\overset{v}{\to}}
\newcommand{\tJ}{\tilde J_z^{n,k}}

\newcommand{\CC}{\mathcal{C}}
\newcommand{\xu}{\overline{\bb x}_p}
\newcommand{\xl}{\underline{\bb x}_p}

\endlocaldefs

\begin{document}

\begin{frontmatter}
\title{Extremes of regularly varying stochastic volatility fields}
%\title{A sample article title with some additional note\thanksref{T1}}
\runtitle{Extremes of regularly varying stochastic volatility fields}
%\thankstext{T1}{A sample of additional note to the title.}

\begin{aug}
%%%%%%%%%%%%%%%%%%%%%%%%%%%%%%%%%%%%%%%%%%%%%%%
%% Only one address is permitted per author. %%
%% Only division, organization and e-mail is %%
%% included in the address.                  %%
%% Additional information can be included in %%
\author[A]{\fnms{Mads}~\snm{Stehr}\ead[label=e1]{mast.fi@cbs.dk}}
\and
\author[A]{\fnms{Anders}~\snm{R{\o}nn--Nielsen}\ead[label=e2]{aro.fi@cbs.dk}}
%%%%%%%%%%%%%%%%%%%%%%%%%%%%%%%%%%%%%%%%%%%%%%
%% Addresses                                %%
%%%%%%%%%%%%%%%%%%%%%%%%%%%%%%%%%%%%%%%%%%%%%%
\address[A]{Department of Finance, Center for Statistics, Copenhagen Business School\printead[presep={,\ }]{e1,e2}}
\end{aug}

\begin{abstract}
We consider a stationary stochastic volatility field $Y_vZ_v$ with $v\in\mathbb{Z}^d$, where $Z$ is regularly varying and $Y$ has lighter tails and is independent of $Z$. We make---relative to existing literature---very general assumptions on the dependence structure of both fields. In particular this allows $Y$ to be non-ergodic, in contrast to the typical assumption that it is i.i.d., and $Z$ to be given by an infinite moving average.

Considering the stochastic volatility field on a (rather general) sequence of increasing index sets, we show the existence and form of a $Y$-dependent extremal functional generalizing the classical extremal index. More precisely, conditioned on the field $Y$, the extremal functional shows exactly how the extremal clustering of the (conditional) stochastic volatility field is given in terms of the extremal clustering of the regularly varying field $Z$ and the realization of $Y$.

Secondly, we construct two different cluster counting processes on a fixed, full-dimensional set with boundary of Lebesgue measure zero: By means of a coordinate-dependent upscaling of subsets, we systematically count the number of relevant clusters with an extreme observation. We show that both cluster processes converge to a Poisson point process with intensity given in terms of the extremal functional.
\end{abstract}



\end{frontmatter}
%%%%%%%%%%%%%%%%%%%%%%%%%%%%%%%%%%%%%%%%%%%%%%
%% Please use \tableofcontents for articles %%
%% with 50 pages and more                   %%
%%%%%%%%%%%%%%%%%%%%%%%%%%%%%%%%%%%%%%%%%%%%%%
%\tableofcontents

%%%%%%%%%%%%%%%%%%%%%%%%%%%%%%%%%%%%%%%%%%%%%%
%%%% Main text entry area:

%%%%%%%%%%%%%%%%%%%%%%%%
\section{Introduction}
\label{sec:intro}
%%%%%%%%%%%%%%%%%%%%%%%%

In this paper we consider, and study the extremal behavior of, a very general stochastic volatility field
\begin{equation}\label{eq:introfield}
	X_v = Y_v Z_v,
	\qquad v\in \Zd ,\qquad d\in\NN,
\end{equation}
which, in the case of $d=1$ plays an important role in econometrics and mathematical finance; see e.g. \cite{FinTimeSeries2009,HandbookVolMod2012} and references therein. Among many other applications they are used for option pricing, insurance mathematics and time series modeling. Allowing for a non-constant volatility can be seen as a very relevant generalization of the classical Black--Scholes setting, since many empirical studies have shown the necessity of this extension: Observed features of implied volatility such as the volatility smile can only be explained by a non-constant volatility; see e.g. \cite{Shephard2009}.

A main assumption in stochastic volatility models is that $(Y_v)$ and $(Z_v)$ are independent. This is in contrast to the popular one-dimensional, i.e. $d=1$, GARCH model, where one of the processes, say, $Z_v$ will feed into future values $Y_{v+k}$, $k\geq 1$ of the other. As both $(Z_v)$ and $(Y_v)$ will in practice be unobserved, the choice between GARCH models or stochastic volatility models will often depend on the varying technical difficulties due the desired application, e.g. estimation and prediction. As we consider a potentially multi-dimensional setting, only the stochastic volatility models have a natural generalization.

So far, most literature on the extremal behavior of stochastic volatility models has focused on the one-dimensional case $d=1$ and furthermore assumed that the volatility process, say $(Z_v)$, is a stationary $\log$-linear Gaussian process, and that the independent noise process $(Y_v)$ consists of i.i.d. random variables with either light- or heavy-tailed marginal distribution; see e.g. \cite{BreidtDavis1998,DavisMikosch2001,DavisMikosch2009,Kulik2011} and references therein. In particular, these scenarios imply that the stochastic volatility process $(X_v)$ does not exhibit extremal clustering no-matter the dependence structure of the underlying $\log$-Gaussian volatility. 
 
A, in some sense, more general set-up than the above is found in \cite{MikoschRezapour2013} where the authors allow the volatility process to have power law tails and then require that the independent noise process has a lighter tail. More specifically, they assume that $(Z_v)$ is a stationary, non-negative and regularly varying process of some index $\alpha>0$, and that the marginal distribution of the i.i.d. noise process $(Y_v)$ satisfy $\EE \abs{Y}^{\gamma}< \infty$ for some $\gamma>\alpha$. Based on the general result \cite[Theorem~2.7]{DavisHsing1995} concerning extremes of stationary processes, they show under standard mixing conditions of the volatility process that extremal clustering of the stochastic volatility process $(X_v)$ is in fact essentially inherited by its volatility process. 

The contribution of the present paper in the literature of stochastic volatility models is twofold. 
First, we consider the extremal behavior of the field $(X_v)$ on the $d$-dimensional discrete index set $D_n\subseteq \Zd$, which is allowed to expand in a very general way. Results on extremes of stationary random fields are to the best of the authors' knowledge mostly formulated under the assumption of $(D_n)$ being a sequence of increasing boxes (\cite{Jakubowski2019,Ling2019,Pereira2006,Soja2019}). However, there are recent papers in which extremes are considered under a more general geometric regime; see for instance \cite{Wintenberger2022,Pereira2017,RNielsenStehr2022}.
In this paper we simply assume that $D_n\subseteq \Zd$ can be approximated from the in- and outside by unions of certain slowly increasing boxes, where the approximations asymptotically have the same size as $D_n$; please see Assumption~\ref{ass:geometry} for details. 

Secondly, and perhaps most importantly, we allow for very general volatility and noise fields in the stochastic volatility model \eqref{eq:introfield}. We assume that the fields $(Y_v)_\vZ$ and $(Z_v)_\vZ$ are independent and stationary, however, in contrast to existing literature, we do not require any of them to be i.i.d. It is important to emphasize that we even allow the field $(Y_v)$ to be non-ergodic, which makes long and infinite memory possible within the framework; see the considerations on ergodicity, long memory and mixing in \cite[Chapters~2 \& 5]{Samorodnitsky2016}.
We require that $(Z_v)$ can be sufficiently approximated by a sequence of random fields $(Z_v^t)$, $t\in\NN$, all of which are assumed to be strongly mixing, regularly varying of index $\alpha>0$, and satisfy an anti-clustering condition in the spirit of e.g. \cite{DavisHsing1995,MikoschRezapour2013}. We also assume that $(Y_v)$ has lighter tails than $(Z_v)$ by requiring that $\abs{Y}^{\gamma}<\infty$ for some $\gamma>\alpha$; see Assumption~\ref{ass:randomfield} for details. Note that we do not assume any of the fields to be non-negative---also in contrast to existing literature. This allow us to model both of the mentioned scenarios from the literature where the volatility field is light-tailed (the volatility being $(Y_v)$) and heavy-tailed (the volatility being $(Z_v)$), respectively. In particular, the stochastic volatility field may in fact show extremal clustering inherited by the noise. This special case not present in the literature can for instance be obtained by having a stationary $\log$-Gaussian volatility field $(Y_v)$ and a regularly varying noise $(Z_v)$.

A stationary random field $(\xi_v)_{v\in\ZZ^d}$, that is regularly varying with index $\alpha$ (the precise definition of regular variation for random fields is given in Section~\ref{sec:prelim}), is said to have extremal index $\theta$ if for a sequence $(a_n)$ of norming constants chosen such that $\abs{D_n}^{-1} \sim \PP(\xi_0>a_n)$ as $n\to\infty$, it holds that
\[
\lim_{n\to\infty}\PP(\max_{v\in D_n}\xi_v\leq a_n x)=\exp(-x^{-\alpha}\theta)
\]
for all $x>0$. Here, $\abs{D_n}$ denotes the cardinality of $D_n$ and $\sim$ denotes asymptotic equivalence. See e.g. \cite{Embrecths1997,Leadbetter1983} for the general definition of the extremal index, beyond the regularly varying case.

Under the assumptions mentioned above on the index sets and fields, our first main result, Theorem~\ref{thm:maxtail}, states the existence (and form) of an \emph{extremal functional} $\eta:\R^{\Rd} \to \R$ satisfying
\begin{equation}\label{eq:intro1}
	\lim_{n\to\infty}
	\PP\bigl(
	\max_{v\in D_n} Y_v Z_v \le a_n x	
	 \mid (Y_v)=(y_v) \bigr)
	= \exp \Bigl( -x^{-\alpha}\eta((y_v)) \Bigr)
\end{equation}
for all $x>0$, where, again, $(a_n)$ is a sequence chosen such that $\abs{D_n}^{-1} \sim \PP(Z_0>a_n)$ as $n\to\infty$.
The term \emph{extremal functional} comes from the fact that, up to a finite and non-vanishing constant, the functional $\eta(\cdot)$ generalizes the classical extremal index. In fact we show that if $(Y_v)$ is just ergodic, $\eta(\cdot)=\eta$ is a constant related to the extremal index $\theta_X$ of the stochastic volatility field $(X_v)$ by
\[
	\theta_X = \frac{\eta}{\lim_{n\to\infty} \abs{D_n}\, \PP(X_0>a_n)} .
\]
We furthermore demonstrate how the extremal functional can be derived using the spectral distribution of $(Z_v)$ and the distribution of $(Y_v)$. 

It is well known in the literature that a regularly varying random field satisfying certain mixing and anti-clustering conditions has a well defined extremal index that furthermore is determined by its spectral distribution; see the seminal paper \cite{DavisHsing1995} for the one-dimensional regularly varying case and \cite{RNielsenStehr2022,Soja2019} for conditions ensuring an extremal index for random fields that are not necessarily regularly varying. In \cite{MikoschRezapour2013} a one-dimensional stochastic volatility model of the form $X_t=Y_tZ_t$ is studied, where $(Z_t)$ is regularly varying and $(Y_t)$ is i.i.d. Assuming that $(Z_t)$ is on the form of an exponential AR(1)-process, a finite moving average or given by a stochastic recurrence equation, it is shown
(under further model-specific assumptions)
that $(X_t)$ satisfies relevant mixing and anti-clustering conditions that ensures the existence of the extremal index. Furthermore, more concrete model-specific formulas for determining the index are given.

The present paper supplies a substantial generalization of the results from \cite{MikoschRezapour2013}: The i.i.d.-requirement on the $(Y_v)$-process is reduced to that of stationarity, and the field $(Z_v)$ simply needs to be suitably approximated by mixing fields obeying to an anti-clustering condition.

Only requiring stationarity of $(Y_v)$ instead of, e.g., ergodicity makes it possible for $(X_v)$ to have long or even infinite memory. This is reflected in the result, where we find a possibly random $(Y_v)$-dependent extremal index, when conditioning on $(Y_v)$. That means that the long/infinite memory coming from $(Y_v)$ in this modeling framework enters in a clear and specified way in the extremal behavior. Another example from the litterature, where the interplay between long memory and extremal theory is studied, can be found in \cite{ChenSamorodnitsky2020}. Here a stable random field with long memory is constructed, and a functional extremal theorem is derived with a limit that is not necessarily the Fr\'echet distribution.

We make assumptions on $(Z_v)$ requiring that it can be approximated by fields exhibiting a sufficient mixing and anti-clustering behavior. Although this set of assumptions at first may look elaborate and specific, it is in fact very general. It covers easily the important example, where $(Z_v)$ is given as an infinite moving average: The field can be approximated by finite moving averages, that clearly satisfies requirements on the dependence structure. On the other hand, the set of assumptions also makes solutions to stochastic recurrence equations, as e.g. the volatility process of a GARCH(1,1) model, possible to handle. Here $(Z_v)$ will in itself meet the necessary requirements regarding mixing and anti-clustering behavior.

As mentioned above, a usual proof technique when showing extremal results for stochastic volatility processes is to show that the full product process $(X_v)=(Y_vZ_v)$ satisfies certain mixing and anti-clustering conditions. The approach in the present paper will be rather different, and this is the reason for the increased generality of the results: We consider $(X_v)$ \emph{conditioned} on $(Y_v)=(y_v)$ and then use a spatial version of the Birkhoff--Khinchin theorem to handle the $(y_v)$-dependence. This makes it possible to only require (approximate) mixing and anti-clustering of $(Z_v)$.

A second set of results (see Section~\ref{sec:mainresultscluster}) concern the asymptotic distribution of a range of cluster counting processes. Here, we assume that $C\subseteq \Rd$ is a fixed, full-dimensional set with boundary of Lebesgue measure zero. We then consider the sequence of index sets $(D_n)$ given by
\begin{equation}\label{fml:multipliedset}
	D_n = (\bb c_n  C) \cap \Zd ,
\end{equation}
where $\bb c_n$ is a $d$-dimensional vector with each element tending to $\infty$, and $\bb c_n C$ denotes the coordinate wise multiplication of $C$ with the elements of $\bb c_n$. First we construct two cluster counting processes on the fixed set $C$ as follows: Let $\Phi_n \subseteq D_n$ denote the (random) set of indices $v\in D_n$ where $X_v=Y_v Z_v$ has an exceedance over the threshold $a_n x$ for fixed $x>0$, i.e.
\[
	\Phi_n = \{v \in D_n\::\: Y_vZ_v>a_n x\} .
\]
Then, according to some \emph{cluster rule}, we make a partition of $\Phi_n$ into a (random) number $\Gamma_n\in\NN \cup \{0\}$ of clusters $\CC_1^n, \dots, \CC_{\Gamma_n}^n$. For any measurable set $A\subseteq C$ we then construct the cluster counting processes by counting the number of clusters $\CC_i^n$ intersecting $\bb c_n A$, that is
\begin{equation}\label{eq:intro2}
	N_n(A) = \sum_{i=1}^{\Gamma_n} 
	\I{\{\CC_i^n \cap \bb c_n A \neq \emptyset\}} .
\end{equation}
The first cluster rule is a spatial version of the classical one-dimensional definition, e.g. seen in \cite{Hsing1988} and \cite{Leadbetter1983a}: We divide $\Zd$ into disjoint (increasing) 
boxes and count a box as a cluster if there is an $a_n x$-exceedance which is within the box and within $D_n$. Typically, for $d=1$, such a counting process is defined on the unit interval $(0,1]$ via a rescaling of the indices, where our geometric
construction in \eqref{fml:multipliedset} allow us to define it on a much more general set $C$. The second cluster counting process, similar to that of \cite{RNielsenStehr2022}, is relatively new in the context of extreme value theory. The process is based on the more intuitive rule that a cluster comprises of the indices in $\Phi_n$ which are relatively close to one another. 
Being close in this setting corresponds to being within a box of the same size as in the first cluster process. For both cluster rules we show that the counting process, conditioned on $(Y_v)=(y_v)$ converges (in the vague topology) towards a homogeneous Poisson process on $C$ with intensity $x^{-\alpha} \eta((y_v))$, where $\eta(\cdot)$ is the extremal functional appearing in \eqref{eq:intro1}; see \cite[Chapter~4]{Kallenberg2017} for details on vague convergence.

Related to the cluster counting process in \eqref{eq:intro2} we define a similar process defined on the original scale in $\Zd$,
\begin{equation*}%\label{eq:intro3}
	L_n(A) = \sum_{i=1}^{\Gamma_n} 
	\I{\{\CC_i^n \cap A \neq \emptyset\}} 
\end{equation*}
for all $A \subseteq \Zd$. Utilizing the convergence above we show (for either cluster rule) the following joint convergence: 
If $(B_n^1),\dots,(B_n^G)$ are disjoint sequences of subsets of $D_n\subseteq\Zd$ each satisfying a slightly modified geometric assumption (see Assumption~\ref{ass:geometry2}), then, conditioned on $(Y_v)=(y_v)$,
\[
	\bigl( L_n(B_n^1),\dots, L_n(B_n^G) \bigr)
	\stackrel{\mathcal{D}}{\to} \bigl(L^1,\ldots,L^G\bigr) ,
\]
where $(L^1,\dots,L^G)$ are independent random variables each being Poisson distributed with parameter $x^{-\alpha}\eta((y_v)) \lim_{n\to\infty} \abs{B_n^g}/\abs{D_n}$ (provided that the limit exists). Here $\stackrel{\mathcal{D}}{\to}$ denotes weak convergence.

The paper is organized as follows. In Section~\ref{sec:prelim} we introduce some notation and relevant concepts needed to understand the assumptions and results of the paper. In Section~\ref{sec:mainresults} we formally introduce the stochastic volatility field and the geometric structures applied in the paper before we present results related to the existence and form of the extremal functional. Section~\ref{sec:mainresultscluster} is devoted to the results on convergence of various cluster counting processes, and Sections~\ref{sec:proofmaxresults}--\ref{sec:garch} are dedicated to proofs. In addition, the paper contains an appendix with various technical results.

%%%%%%%%%%%%%%%%%%%%%%%%
\section{Preliminaries}
\label{sec:prelim}
%%%%%%%%%%%%%%%%%%%%%%%%

Before presenting the first set of main results in the next section, we introduce relevant concepts and notation.

We let $\abs{{}\cdot{}}$ be a general size-measure understood as follows: $\abs{v}$ is the Euclidean norm of a single one- or multidimensional point $v$, $\abs{A}$ is Lebesgue measure of a full-dimensional set $A\subseteq \Rd$, and $\abs{A}$ is the number of points in a discrete set $A\subseteq \Zd$. Moreover, we let $\norm{{}\cdot{}}$ be the $\max$ norm associated to the Euclidean norm, that is, for a finite set $A \subseteq \Rd$ and a sequence of real numbers $(x_v)_{v\in A}$, 
\[
	\norm{(x_v)_{v\in A}} = \max_{v\in A} \,\abs{x_v} .
\]

For a number $y\in\R$ we let $y_+=y \I{\{y\ge 0\}}$ and $y_-=\abs{y} \I{\{y< 0\}}$ denote the positive and negative part, respectively. Moreover, for $\gamma>0$ we write $y_+^\gamma = (y_+)^\gamma$ and $y_-^\gamma = (y_-)^\gamma$.

We will use the following notation for a coordinate-wise scaling of a set $A \subseteq \Rd$: With $\bb{r}$ (in bold) denoting a vector of $d$ elements $\bb{r}=(r_1,\dots,r_d)$ we write
\[
	\bb{r} A = \{ (r_1 a_1,\dots,r_d a_d) \in \Rd\::\: (a_1,\dots,a_d)\in A\} . 
\]
That is, $\bb{r} A$ is a compact notation for the linear transformation of $A$ induced by the diagonal matrix with entries $r_1,\dots,r_d$. For such a vector $\bb r$, we let $\bb r^{-1}=1/\bb r$ be the vector of reciprocals, and
we let $\bb r^*\in\R$ denote the product of its entries, i.e. $\bb r^* = r_1\cdot\ldots\cdot r_d$. In particular we obtain for any full-dimensional set $A$ that $\bb{r}A$ has Lebesgue measure
\begin{equation*}%\label{eq:matrixprodleb}
	\abs{\bb{r}A} =   \abs{\bb r^*}\, \abs{A} =  \abs[\Big]{\prod_{\ell=1}^d r_\ell}\, \abs{A}  .
\end{equation*}
Moreover, for two such vectors $\bb r$ and $\bb s$, we write $\bb r \le \bb s$ if $r_\ell \le s_\ell$ for each $\ell = 1,\dots,d$. Lastly, when considering a sequence $(\bb r_n)_{n\in\NN}$ of vectors $\bb r_n = (r_{n,1},\dots,r_{n,d})$, we write $\bb r_n \to \infty$ if $r_{n,\ell}\to \infty$ as $n\to\infty$ for all $\ell$. Similarly, $\bb r_n = o(\bb s_n)$ for two sequences of vectors if $r_{n,\ell}/s_{n,\ell} \to 0$ for each $\ell$.

For any $\vZ$ and $m\in\NN$ we define
\begin{align*}
	\Bvm 
	= \bigl(v+[-m,m]^d \bigr)\cap\Zd
	= \{z \in \Zd \::\: \norm{z-v} \le m \}
\end{align*}
and, in particular, we use the notation $\Bm = B^{(m)}_0$ where $0$ denotes the origin of $\Rd$. We say that a random field $(R_v)_{v\in\Zd}$ is regularly varying with index $\alpha>0$ if all finite subsets are regularly varying. This means that for all $m\in\NN$, defining $\bb{R}^m=(R_v)_{v\in B^{(m)}}$, there exists a random vector $\bb{\Theta}^{m}=(\Theta^m_v)_{v\in \Bm}$ with values in
\[
	\mathbb{S}^{\abs{\Bm}-1}
	=\{\bb{x}/\norm{\bb{x}}\::\: \bb{x}\in \R^{\Bm}\setminus\{\bb 0\}\},
\]
such that
\begin{equation}\label{eq:regvarmulti}
	\frac{\PP(\norm{\bb{R}^m}>xt,\bb{R}^m/\norm{\bb{R}^m}\in{}\cdot{} )}
	{\PP(\norm{\bb{R}^m}>t)}\stackrel{v}{\to}x^{-\alpha}\PP(\bb{\Theta}^{m}\in{}\cdot{})
\end{equation}
as $t\to\infty$. Here, $\stackrel{v}{\to}$ denotes \emph{vague convergence}; details on this type of convergence can be found in \cite[Chapter~4]{Kallenberg2017}. Typically, the distribution of $\bb{\Theta}^{m}$ is referred to as \emph{the spectral measure} of the field $\bb R^m$, and we will correspondingly refer to $\bb{\Theta}^{m}$ as the \emph{spectral field}.

In \cite{Wu2020} it is demonstrated that regular variation of a random field can equivalently be formulated as a property of the entire field and not just based on the family of finite dimensional fields. That would lead to a spectral random field $\bb{\Theta}=(\Theta_v)_{v\in\Zd}$, similar but not identical to the finite dimensional counterparts $\bb{\Theta}^m$. However, in our framework, where it assumed that $(Z_v)$ can be approximated by fields exhibiting less dependence, the extremal behavior of the approximating random fields will typically be determined from finitely many random variables. Therefore, the formulation using finite dimensional distributions is more suitable.

Lastly, we need the concept of strong mixing (also known as $\alpha$-mixing) of a random field $(R_v)_{v\in\Zd}$. We define it essentially as in \cite{Doukhan1994} however based on a coordinate-dependent distance (the vector $\bb h$ below). We refer to \cite{Doukhan1994} for general definitions, examples and relations between strong mixing and other mixing definitions. Let for all $A\subset\Zd$  the set $\sigma_R(A)=\sigma(R_v\::\: v\in A\}$ be the $\sigma$-field generated by all variables with an index in~$A$. Furthermore, for $\bb{h}=(h_1,\ldots,h_d)\in (0,\infty)^d$ we say that sets $A,B\in \Zd$ are $\bb{h}$-separated, if
\[
	\abs{b_\ell-a_\ell} > h_\ell ,\qquad \ell=1,\ldots,d
\]
for all points $(a_1,\ldots,a_d)\in A$ and $(b_1,\ldots,b_d)\in B$.

\begin{definition}[Strong mixing]
A random field $(R_v)_{\vZ}$ is said to be strong mixing with mixing parameters $(\alpha( \bb h))_{ \bb h \in (0,\infty)^d}$ in $(0,\infty)$, if for all $A,B\subset\Zd$, where $A$ and $B$ are $\bb h$-separated, it holds that
\[
	\sup_{\mathcal A \in \sigma_R(A),\mathcal B \in \sigma_R(B)}
	\abs[\big]{
	\PP(\mathcal A\cap \mathcal B)-\PP(\mathcal A)\PP(\mathcal B)
	}
	\le \alpha(\bb h),
\]
and $\alpha(\bb h)\to 0$, whenever $\bb h\to\infty$. 
\end{definition}

%%%%%%%%%%%%%%%%%%%%%%%%
\section{Results on the extremal functional}
\label{sec:mainresults}
%%%%%%%%%%%%%%%%%%%%%%%%

In this section we formally present the underlying assumptions of the paper after
which we can present the first set of main results.

For the sequence $(D_n)$ of index sets under consideration, we assume that they increase and can be sufficiently approximated by a union of slowly increasing boxes. To formally define this, let $\bb t_n=(t_{n,1},\ldots,t_{n,d})$ be a sequence in $\NN^d$ of scaling vectors, and let $x_n=(x_{n,1},\ldots,x_{n,d})$ with $x_{n,\ell}\le t_{n,\ell}$ for $\ell=1,\ldots,d$ be a sequence in $(0,\infty)^d$ of translation (or shift) vectors. For each $n\in\NN$, we then define the system of boxes $(J_z^{\bb t_n,x_n})_{z\in\Zd}$ by
\[
	J_z^{\bb t_n,x_n}
	=
	\Bigl( x_n + \bb t_n \bigl(z+[0,1)^d \bigr) \Bigr)
	\cap \Zd.
\]
For a sequence $(D_n)_{n\in\NN}$ of sets in $\Zd$ we define the approximations
\[
	D_{n,\bb t_n,x_n}^-
	=
	\bigcup_{J^{\bb t_n,x_n}_z\subseteq D_n}
	J^{\bb t_n,x_n}_z
	\qquad\text{and}\qquad 
	D_{n,\bb t_n,x_n}^+
	=
	\bigcup_{J^{\bb t_n,x_n}_z\cap D_n\neq \emptyset}
	J^{\bb t_n,x_n}_z,
\]
which obviously satisfy
\[
	D_{n,\bb t_n,x_n}^-
	\subseteq D_n
	\subseteq D_{n,\bb t_n,x_n}^+.
\]
As the choice of scalings $(\bb t_n)$ and translations $(x_n)$ will be given and fixed by Assumption~\ref{ass:geometry} below, we will write $J_z^n$, $D_n^-$ and $D_n^+$ for ease of notation.

Furthermore, we say that a sequence $(D_n)$ is regular if there exists an increasing sequence of boxes $D_n'$, i.e. on the form $\bigtimes_\ell[a_\ell,b_\ell]\cap \Zd$,
such that for all $n$ it holds that $D_n\subseteq D_n'$ and $\abs{D_n'}\leq c\,\abs{D_n}$ uniformly in $n$ for some finite constant $c$.
For a regular sequence $(D_n)$, we write $D_n\to\Zd$, if the corresponding sequence $(D_n')$ satisfies
\[
	\bigcup_{n=1}^\infty D_n'=\Zd.
\]
Note in particular that $\abs{D_n}\to\infty$ for a regular sequence $(D_n)$ satisfying $D_n\to\Zd$ as $n\to\infty$.

\begin{assumption}\label{ass:geometry}
The sequence of index sets $(D_n)_{n\in\NN}$ in $\Zd$ is regular with $D_n~\to~\Zd$, 
and there exist $(\bb t_n)$ and $(x_n)$ with $\bb t_n \to \infty$ and $\bb t_n^* = o(\abs{D_n})$,
such that
\begin{equation}\label{eq:geometricassumption}
	\frac{\abs{\dnp}-\abs{\dnm}}{\abs{D_n}} \to 0
\end{equation} 
as $n\to\infty$.
\end{assumption}
Assumption~\ref{ass:geometry} can be seen as a relaxation of the geometrical assumption on the sequence of index sets in \cite[Section~2]{RNielsenStehr2022}, where a certain type of convexity is required. We let $\preceq$ be an arbitrary translation invariant total order on $\Zd$ (e.g. the lexicographical order) and define
\[
\Av=\{z\in v+(\bb{t}_n[-1,1]^d\cap\Zd)\,:\,v\prec z\}
\] 
and
\[
\Avm=\{z\in v+([-m,m]^d\cap\Zd)\,:\,v\prec z\},
\]
see \cite{RNielsenStehr2022} for similar definitions.

\begin{assumption}\label{ass:randomfield}
We say that the stochastic volatility field $(Y_vZ_v)_\vZ$ satisfies Assumption~\ref{ass:randomfield} if all of the following conditions are satisfied:

The field $(Z_v)_\vZ$ is stationary and it satisfies the tail-balance condition
\begin{equation}\label{eq:tailbalance}
	\lim_{x\to \infty}	
	\frac{\PP(Z_0 < -x)}{\PP(Z_0>x)}
	= \frac{1-p}{p}
\end{equation}
for some $p \in (0,1]$. Furthermore, it can be approximated by a sequence of fields $(Z_v^t)_\vZ$, $t\in\NN$, with the following properties (some of which are relative to the geometry given by Assumption~\ref{ass:geometry}):
\begin{enumerate}[label=\normalfont(\roman*)]
	\item \label{eq:fieldass1}
	$(Z_v^t)_\vZ$ is stationary and regularly varying with index $ \alpha>0$,
	\item \label{eq:fieldass2}
	$(Z_v^t)_{\vZ}$ is strongly mixing with mixing parameters $(\alpha^t(\bb{h}))_{\bb{h}\in(0,\infty)^d}$ such that there exists a  sequence $(\bb \gamma_n^t)$ with $\bb \gamma_n^t\to\infty$ and $\bb \gamma_n^t=o(\bb t_n)$, where
%\[
%	\alpha^t(\bb \gamma_n^t)\, \abs{D_n}/\abs{\Jz} \to 0,
%\]
\[
	\alpha^t(\bb \gamma_n^t)\, \abs{D_n}/\bb t_n^* \to 0,
\]
	\item \label{eq:fieldass3}
	$(Z_v^t)_{v\in \Zd}$ satisfies the anti-clustering condition 
\[
	\lim_{m\to\infty}\limsup_{n\to\infty}
	\PP\Bigl(
	\max_{\An\setminus \Anm} \abs{Z_v^t} > s a_n\mid \abs{Z_0^t}>sa_n \Bigr)=0
\]
for all $s>0$, where $\abs{D_n}^{-1}\sim \PP(Z_0^t>a_n)$,
	\item \label{eq:fieldass4}
	there exist $C_t^+>0$ and $C_t^->0$ such that
	\begin{equation}\label{eq:approx1}
		\lim_{x\to \infty} \frac{\PP(Z_0^t>x)}{\PP(Z_0>x)} = C_t^+
		\qquad\text{and}\qquad
		\lim_{x\to \infty} \frac{\PP(Z_0^t<-x)}{\PP(Z_0<-x)} = C_t^-,
	\end{equation}
	where $C_t^+$ and $C_t^-$ additionally satisfy	
	\begin{equation}\label{eq:C+C-}
		\lim_{t\to \infty} C_t^+ = \lim_{t\to \infty} C_t^-=1 ,
	\end{equation}
		\item \label{eq:fieldass5} and lastly
	\begin{equation}\label{eq:approx2}
	\begin{aligned}
		\liminf_{t\to\infty}\liminf_{x\to\infty} 
		\frac{
		\PP( Z_0^t>x, Z_0 > x)}{\PP(Z_0>x)} %\\ &
		=
		\liminf_{t\to\infty}\liminf_{x\to\infty} 
		 \frac{
		\PP( Z_0^t<-x, Z_0<-x)}{\PP(Z_0<-x)} %\\ &
		= 1 .
	\end{aligned}
	\end{equation}
\end{enumerate}

The field $(Y_v)_\vZ$ is stationary and independent of $(Z_v)$ and all the approximating fields $(Z_v^t)$, $t\in\NN$. Moreover
\begin{equation}\label{eq:gammaY}
		\exists\ \gamma > \alpha:\quad \EE \abs{Y_0}^\gamma < \infty .
\end{equation}
\end{assumption}

Note that \ref{eq:fieldass1} and \ref{eq:fieldass3} above imply that the right tail of $Z_v$, $\vZ$, is (marginally) regularly varying with index $\alpha$. In particular, there exists a sequence of norming constants $(a_n)_{n\in\NN}$ such that
\[
	\abs{D_n} \PP(Z_0>a_n x) \to x^{-\alpha}
\]
for all $x>0$ as $n\to\infty$. Since $Z_0$ is regularly varying this is equivalent to the asymptotic equality $\abs{D_n}^{-1} \sim \PP(Z_0>a_n)$. This particular sequence $(a_n)$ will be used in the results below, and with this choice of sequence, \ref{eq:fieldass3} will still be satisfied for each $(Z_v^t)$ because of \ref{eq:fieldass4}.

There are many cases in which certain elements of the assumption are obviously satisfied. For example, if the field $(Z_v)$ itself satisfies the regular variation, mixing, and anti-clustering required by $(Z_v^t)$, one can choose $Z_v=Z_v^t$ for all $\vZ$ and $t\in\NN$. Furthermore, if the approximating fields $(Z_v^t)$ are $m_t$-dependent for some $m_t$, that is, if the families $\{Z_v^t\::\ v\in A\}$ and $\{Z_v^t\::\ v\in B\}$ are independent for all $m_t$-separated sets $A,B$, then  
\ref{eq:fieldass2} and \ref{eq:fieldass3} are trivially satisfied. For instance, if $(Z_v)$ is so-called $\max$-$t$-approximable as introduced in \cite{Jakubowski2019}, the approximating fields are $2t$-dependent and furthermore \ref{eq:fieldass4}--\ref{eq:fieldass5} are satisfied.

\begin{example}\label{ex:movingaverage}
Let $(\xi_v)_{v\in \Zd}$ be an i.i.d. field of regularly varying random variables of index $\alpha>0$ which satisfy
\[
	\lim_{x\to \infty}	
	\frac{\PP(\xi_0 < -x)}{\PP(\xi_0>x)}
	= \frac{1-p_\xi}{p_\xi}
\]
for some $p_\xi \in (0,1]$. For a field of constants $(\psi_u)_{u\in\Zd}$ satisfying
\[
	\sum_{u \in \Zd} \abs{\psi_u}^{\delta} < \infty
\]
for some $\delta \in (0,\alpha) \cap (0,1]$, the moving average field $(Z_v)_\vZ$ given by
\begin{equation}\label{eq:movingaverage}
	Z_v = \sum_{u\in\Zd} \psi_u \xi_{v-u}
\end{equation}
satisfies Assumption~\ref{ass:randomfield} with approximating field $(Z_v^t)$ given as
\begin{equation*}
	Z_v^t = \sum_{\norm{u}\le t} \psi_u \xi_{v-u} 
\end{equation*}
for all $t\in\NN$.
That $(Z^t_v)$ satisfies \ref{eq:fieldass2} and \ref{eq:fieldass3} in Assumption~\ref{ass:randomfield} follows since it is $2t$-dependent. This example is described in details in Section~\ref{sec:movingaverage}.
\end{example}

\begin{example}\label{ex:garch}
In a one-dimensional setting, i.e. $d=1$, we consider the process $(Z_v)_{v\in \ZZ}$ defined recursively by
\begin{equation}\label{eq:garchdef}
Z_v^2=\alpha_0+Z_{v-1}^2(\alpha_1\xi_{v-1}^2+\beta_1),
\end{equation}
where $(\xi_v)_{v\in\ZZ}$ is an i.i.d. sequence with mean 0 and variance 1, and $\alpha_0$, $\alpha_1$ and $\beta_1$ are positive constants with $\alpha_1+\beta_1<1$. Though we apply $(Z_v)_{v \in \ZZ}$ directly in the modeling, this process is often used as the volatility process for a GARCH(1,1) process. That is, the process $(\tilde{X})_{v\in\ZZ}$ defined as
\[
	\tilde{X}_v=\xi_v Z_v.
\]
As opposed to our model (\ref{eq:introfield}), in the GARCH(1,1) process there is dependence between the factors defining $\tilde{X}_v$. 

Defining $A=\alpha_1\xi_1^2+\beta_1$, we assume that
\begin{enumerate}
\item  $\xi_1$ has a strictly positive density on $\R$,
\item $\EE\log A<0$,
\item there exists $\kappa_0$ such that $\EE A^{\kappa_0}>1$ and $\EE A^{\kappa_0}\log (A_+)<\infty$,
\item the solution (the solution exists; c.f. \cite[Remark~2.5]{Basrak2002}) $\alpha$ to the equation $\EE A^\alpha=1$ is not an even integer.
\end{enumerate}
Then, with results that are based on a result in \cite{Kesten1973}, it follows from \cite[Theorem~2.4, Remark~2.5, Theorem 3.1 and Corollary 3.5]{Basrak2002} that there exists a strictly stationary causal solution to the equation (\ref{eq:garchdef}), such that  $(Z_v)_{v\in\ZZ}$ is a regularly varying process with index $2\alpha$ that is strongly mixing (in fact with geometrically decaying rate). Furthermore, due to \cite[Lemma~4.3]{MikoschRezapour2013} it satisfies the anti-clustering condition (iii) from Assumption~\ref{ass:randomfield}. This example is described in further details in Section~\ref{sec:garch}. More generality but less clarity could have been obtained by simply requiring that $(Z_v)$ satisfies the conditions given in \cite{Kesten1973} and is furthermore strongly mixing.

More generally, a stochastic process $(Z_v)_{v\in\ZZ}$ given as the solution to a stochastic recurrence equation on the form
\[
	Z_v=A_vZ_{v-1}+B_v,
\]
where $((A_v,B_v))_{v\in\ZZ}$ form an i.i.d. sequence of non-negative variables, will be stationary, regularly varying, strongly mixing and satisfy the anti-clustering condition, if it satisfies the conditions of \cite{Kesten1973} and is strongly mixing with geometric rate; see \cite[Theorem~4.4]{MikoschRezapour2013}.
\end{example}

In the definition of the functional $\eta$ in Theorem~\ref{thm:maxtail} below, we use notation relating to the process $\bb{Y}=(Y_v)_{v\in\Zd}$ and its distribution. In particular, $\EE^{\bb{Y}}({}\cdot{}\mid \hat{\Ii})$ denotes the conditional expectation on the value space of all real-valued fields $(x_v)_{v\in\Zd}$ equipped with the natural Borel $\sigma$-algebra and the probability measure $\bb{Y}(\PP)$, conditioned on the translation (translations in all directions) invariant $\sigma$-algebra $\hat{\Ii}$. Additionally, $\hat{\bb{Y}}=(\hat{Y}_v)_{v\in\Zd}$ denotes the identity map. Details are given in Appendix~\ref{app:ergodictheorem}.

\begin{theorem}\label{thm:maxtail}
Let $(D_n)_{n\in\NN}$ be a sequence of sets in $\Zd$ satisfying Assumption~\ref{ass:geometry}. Let $(Y_vZ_v)_\vZ$ be a stochastic volatility field satisfying Assumption~\ref{ass:randomfield}, and choose the norming constants $(a_n)_{n\in\NN}$ such that $\abs{D_n}^{-1} \sim \PP(Z_0>a_n)$. Then, for almost all realizations $(y_v)$ of $(Y_v)$, it holds that
\begin{equation}\label{eq:main1}
	\lim_{n\to\infty}
	\PP\bigl(
	\max_{v\in D_n} Y_v Z_v \le a_n x	
	 \mid (Y_v)=(y_v) \bigr)
	= \exp \Bigl( -x^{-\alpha}\eta((y_v)) \Bigr)
\end{equation}
for all $x>0$. Here the functional $\eta:\R^{\Rd} \to [0,\infty)$ is given (and is well defined) by the limit
\[
	\eta((y_v))=\lim_{t\to\infty}\lim_{K\to\infty}\lim_{m\to\infty}\eta^{t,K,m}((y_v))
\]
with
\begin{equation*}%\label{eq:main1eta}
\begin{aligned}
	& \eta^{t,K,m}((y_v)) \\ & \quad 
	=\frac{\EE^{\bb{Y}}\Bigl[\EE\Bigl(\bigl((Y_0\I{\{\abs{Y_0}\le K\}}\Theta^{t,m}_0)_+^\alpha
	-(\underset{v\in A^{(m)}_0}{\max}Y_v\I{\{\abs{Y_v}\le K\}}\Theta^{t,m}_v)_+^\alpha\bigr)_+\mid\bb{Y}=\hat{\bb{Y}}\Big)\mid\hat{\Ii}\Bigr]((y_v))}{\EE(\Theta_0^{t,m})_+^\alpha},
\end{aligned}
\end{equation*}
where  $(\Theta^{t,m}_v)_{v\in B^{(m)}}$ represents the spectral field of $(Z^t_v)_{v\in B^{(m)}}$.
\end{theorem}

Corollary~\ref{cor:maxtail2} below demonstrates that the extremal functional $\eta$ satisfies properties similar to properties met by the classical extremal index under sufficient mixing and anti-clustering conditions. See e.g. the discussions in \cite[Section~5]{Jakubowski2019}, \cite[Section~4]{Soja2019} and \cite[Section~4]{Wu2020} that all addresses extremal indices for random fields.

\begin{corollary}\label{cor:maxtail2}
Let the assumptions of Theorem~\ref{thm:maxtail} be satisfied. Then the functional $\eta:\R^{\Rd}\to [0,\infty)$ can equivalently be found as
\begin{equation*}%\label{eq:main2theta}
\begin{aligned}
	\eta((y_v)) & = \lim_{n\to\infty}  
	\sum_{\Jz \subseteq D_n}\PP\bigl(\max_{v \in \Jz} y_v Z_v > a_n\bigr) \\ &
	= \lim_{n\to\infty}
	\sum_{\Jz \cap D_n \neq \emptyset}\PP\bigl(\max_{v \in \Jz} y_v Z_v > a_n\bigr)  \\ &
	= \lim_{n\to\infty}
	\sum_{v\in D_n}\PP\bigl(\max_{u \in \Av} y_u Z_u \le a_n < y_v Z_v  \bigr)
	\end{aligned}
\end{equation*}
as $n\to\infty$.
\end{corollary}

Next we consider the special case, where $(Y_v)$ is ergodic; see the precise definition in Appendix~\ref{app:ergodictheorem}. If the field $(Y_v)$ is ergodic, Theorem~\ref{thm:maxtail} provides the existence and specific form of the extremal index of the stochastic volatility field $(Y_v Z_v)$. To clarify this, we note that
\[
	\frac{\PP(Y_0 Z_0 >x)}{\PP(Z_0 > x)} \to
	\EE (Y_0)_+^\alpha + \frac{1-p}{p}\EE (Y_0)_-^\alpha
\]
as $x\to\infty$ (see for instance \cite{Breiman}), which in particular implies that
\begin{equation}\label{eq:tailofY0Z0}
	\abs{D_n}\, \PP(Y_0 Z_0>a_n x)
	\to x^{-\alpha} 
	\Bigl( \EE (Y_0)_+^\alpha + \frac{1-p}{p}\EE (Y_0)_-^\alpha \Bigr)
\end{equation}
as $n\to\infty$ for all $x>0$.

\begin{corollary}\label{cor:maxtailergodic}
Let the assumptions of Theorem~\ref{thm:maxtail} be satisfied, and assume in addition that the field $(Y_v)_\vZ$ is ergodic. Then $\eta^{t,K,m}((y_v))$ from Theorem~\ref{thm:maxtail} can be chosen as the $(y_v)$-independent constant
\[
	\eta^{t,K,m}=\frac{\EE\Bigl(\bigl((Y_0\I{\{\abs{Y_0}\le K\}}\Theta^{t,m}_0)_+^\alpha-(\max_{v\in A^{(m)}_0}Y_v\I{\{\abs{Y_v}\le K\}}\Theta^{t,m}_v)_+^\alpha \bigr)_+\Bigr)}{\EE(\Theta_0^{t,m})_+^\alpha},
\]
and, consequently, $\eta((y_v))=\eta$ can be chosen as the constant
\[
\eta=\lim_{t\to\infty}\lim_{K\to\infty}\lim_{m\to\infty}\eta^{t,K,m}.
\]
In particular, $(Y_v Z_v)_\vZ$ has extremal index $\theta$ given by
\[
	\theta = \frac{\eta}{\EE (Y_0)_+^\alpha + \frac{1-p}{p}\EE (Y_0)_-^\alpha } ,
\]
i.e.
\[
	\lim_{n\to\infty} 
	\PP\bigl(\max_{v\in D_n} Y_v Z_v \le a_n x \bigr)
	= \exp\Bigl( -x^{-\alpha} \bigl( \EE (Y_0)_+^\alpha + \tfrac{1-p}{p}\EE (Y_0)_-^\alpha \bigr)\, \theta \Bigr)
\]
for all $x>0$.
\end{corollary}

\begin{corollary}\label{cor:mtdpendence}
Let the assumptions of Theorem~\ref{thm:maxtail} be satisfied, and assume in addition that the approximating fields $(Z_v^t)_\vZ$ are $m_t$-dependent for some $m_t$ for each $t\in\NN$. Then the results of Theorem~\ref{thm:maxtail} and Corollaries~\ref{cor:maxtail2}--\ref{cor:maxtailergodic} hold with
\[
	\eta((y_v))=\lim_{t\to\infty}\lim_{m\to\infty}\eta^{t,m}((y_v))
\]
where
\begin{align*}
%	\MoveEqLeft
	\eta^{t,m}((y_v)) %\\ & 
	=\frac{\EE^{\bb{Y}}\Bigl[\EE\Bigl(\bigl((Y_0\Theta^{t,m}_0)_+^\alpha
	-(\underset{v\in A^{(m)}_0}{\max}Y_v \Theta^{t,m}_v)_+^\alpha\bigr)_+\mid\bb{Y}=\hat{\bb{Y}}\Big)\mid\hat{\Ii}\Bigr]((y_v))}{\EE(\Theta_0^{t,m})_+^\alpha} .
\end{align*}
\end{corollary}

\begin{corollary}\label{cor:Zisstrongmixing}
If the field $(Z_v)$ itself satisfies \eqref{eq:tailbalance} and \ref{eq:fieldass1}--\ref{eq:fieldass3} of Assumption~\ref{ass:randomfield}, and $(Y_v)$ is a stationary random field, independent of $(Z_v)$, satisfying \eqref{eq:gammaY}, then the results of Theorem~\ref{thm:maxtail} and Corollaries~\ref{cor:maxtail2}--\ref{cor:maxtailergodic} hold with
\[
	\eta((y_v))=\lim_{K\to\infty}\lim_{m\to\infty}\eta^{K,m}((y_v))
\]
and
\begin{align*}
	&\eta^{K,m}((y_v))\\
	&=\frac{\EE^{\bb{Y}}
	\Bigl[\EE\Bigl(\bigl((Y_0\I{\{\abs{Y_0}\le K\}}\Theta^{m}_0)_+^\alpha
	-(\max_{v\in A^{(m)}_0}Y_v\I{\{\abs{Y_v}\le K\}}\Theta^{m}_v)_+^\alpha \bigr)_+\mid\bb{Y}=\hat{\bb{Y}}\Bigr)\mid\hat{\Ii}\Bigr]((y_v))}
	{\EE(\Theta_0^{m})_+^\alpha},
\end{align*}
where $(\Theta^{m}_v)_{v\in B^{(m)}}$ represents the spectral field of $(Z_v)_{v\in B^{(m)}}$.
\end{corollary}

\setcounter{example}{0}
\begin{example}[Continued]
When $(Z_v)$ is the moving average defined by \eqref{eq:movingaverage} and $(Y_v)$ is ergodic and satisfies the requirements of Assumption~\ref{ass:randomfield}, the stochastic volatility field $(Y_v Z_v)$ has extremal index
\[
	\theta = \frac{\EE \bigl[
		\max_{\vZ} \bigl( p_\xi (Y_v\psi_v)_+^\alpha 
		+ (1-p_\xi) (Y_v\psi_v)_-^\alpha \bigr)_+
	\bigr]}{p\, \norm{\psi}_\alpha^\alpha},
\]
where
\begin{equation}\label{eq:normpsi}
	\norm{\psi}_\alpha^\alpha = \sum_{u\in\Zd} \abs{\psi_u}^\alpha
\end{equation}
and
\[
	p = \sum_{u\in\Zd} \bigl(p_\xi (\psi_u)_+^\alpha + (1-p_\xi) (\psi_u)_-^\alpha\bigr)/\norm{\psi}_\alpha^\alpha .
\]
The specific form of the extremal functional $\eta(\cdot)$ in the general case of a non-ergodic $(Y_v)$ can be found in Section~\ref{sec:movingaverage}.
\end{example}

\begin{example}[Continued]
In Section~\ref{sec:garch} it is demonstrated how Corollary~\ref{cor:Zisstrongmixing} can be applied to determine the extremal functional in the one-dimensional case, where $(Z_v)_{v\in\ZZ}$ is given by (\ref{eq:garchdef}).
\end{example}

%%%%%%%%%%%%%%%%%%%%%%%%
\section{Results on cluster counting processes}
\label{sec:mainresultscluster}
%%%%%%%%%%%%%%%%%%%%%%%%

In this section we consider only index sets $D_n$ given as follows:
Let $C\subseteq\Rd$ be a bounded set with non-empty interior $\ii C \neq \emptyset$ satisfying that the boundary has Lebesgue measure zero in the sense that $\abs{\cl C\setminus \ii C}=0$. Without loss of generality we may assume that $C$ has unit volume $\abs{C}=1$. Let $\bb c_n \to \infty$ and $\bb k_n\to\infty$ be any pair of sequences of vectors satisfying $\bb k_n = o(\bb c_n)$, and define
\begin{equation*}%\label{eq:Dnclusterprocess}
	D_n = \bb c_n C \cap \Zd
\end{equation*}
for all $n\in\NN$. 
Let $\bb t_n$ be given by its entries $t_{n,\ell} = \lfloor c_{n,\ell} / k_{n,\ell} \rfloor$, where $\lfloor \cdot \rfloor$ is the floor function. Then the sequence $(D_n)$ satisfies Assumption~\ref{ass:geometry} with this choice of scalings $\bb t_n$ and arbitrary  translations $x_n$. In particular, the approximation of $D_n$ is constructed with boxes $\Jz\subseteq \Zd$ of side-lengths $t_{n,\ell}$.

As mentioned in the introduction, we will construct two counting processes,
$N_n$ and $\tilde N_n$, each counting the number of clusters of indices at which the stochastic volatility field $(Y_vZ_v)_\vZ$ exceeds the threshold $a_n x$.
Here, as in the previous section, the norming constants $(a_n)$ are chosen such that
\[
	\abs{D_n} \PP(Z_0>a_n) \to 1
	\qquad \text{or equivalently}
	\qquad
	\bb c_n^* \PP(Z_0>a_n) \to 1 
\]
as $n\to\infty$. According to a given cluster rule, both counting processes are constructed by making a partition of the random set
\[
	\Phi_n = \{ v \in D_n\::\: Y_v Z_v > a_n x\} 
\]
into a random number of clusters.

First we define the cluster measure $N_n$ on $C$ by counting the box $\Jz$ as a cluster if there is an $a_n x$-exceedance within $\Jz \cap D_n$, and if such an exceedance lies in the $\bb c_n$-upscaled set at hand,~i.e.
\begin{equation*}%\label{eq:Nnmeasure}
	N_n(A) = \sum_{z\in\Zd} 
	\I{\{\Jz \cap \Phi_n \cap \bb c_n A \neq \emptyset\}}
\end{equation*}
for all $x>0$ and all $A\in\Bb(C)$. Note that this indeed has the general form introduced in \eqref{eq:intro2} with well-defined clusters $\Jz \cap \Phi_n$.

Secondly we define the cluster measure $\tilde N_n$ on $C$, which is based on the more intuitive interpretation of a cluster as being a collection of indices sufficiently close to one another at which there is an exceedance of the threshold $a_n x$. 
We then say that the elements $u_1,\dots,u_R \in \Phi_n$ form a cluster if (possibly after a reordering) it for all $i=2,\dots,R$ holds that
\[
	\norm{\bb t_n^{-1}(u_i-u_{i-1})} < 1 ,
\]
and if
\[
	\norm{\bb t_n^{-1}(u_i-v)} \ge 1 
\]
for all $i=1,\dots,R$ and all $v \in \Phi_n \setminus\{u_1,\dots,u_R\}$. Note that this cluster rule is comparable to the first one, as it simply collects all indices with an $a_n x$-exceedances which can be covered by an overlapping sequence of non-empty $J^n$-boxes.

If $\Phi_n \neq \emptyset$, this uniquely partitions $\Phi_n$ into $\Gamma_n\ge 1$ disjoint clusters $\CC_1^n,\dots,\CC_{\Gamma_n}^n$ (with arbitrary ordering). We now define the random point measure $\tilde N_n$ on $C$ as the number of clusters $\CC_i^n$ intersecting the $\bb c_n$-upscaled set at hand, i.e. 
\begin{equation*}%\label{eq:tildeNnmeasure}
	\tilde N_n(A) = 
	\sum_{i=1}^{\Gamma_{n}}\I{\{\CC_i^n \cap \bb c_n A \neq \emptyset\}}
\end{equation*}
for all $A\in\Bb(C)$, with the convention that $\tilde{N}_n(A)=0$ if $\Phi_n=\emptyset$ (that is, $\Gamma_n=0$).

For the Theorems~\ref{thm:clusterconvergence} and \ref{thm:originalscale} below, we refer to Theorem~\ref{thm:maxtail} and Corollary~\ref{cor:maxtailergodic} for the definition of the extremal functional $\eta((y_v))$ and the corresponding constant $\eta$ in the case of ergodic $(Y_v)$. However, in a number of special cases, it may be more convenient to consult Corollaries \ref{cor:mtdpendence}--\ref{cor:Zisstrongmixing} for, in the relevant situation, equivalent definitions.

\begin{theorem}\label{thm:clusterconvergence}
Let $C\subseteq\Rd$ and $(D_n)_{n\in\NN}$ be as above, and let $(Y_v Z_v)_\vZ$ be a stochastic volatility field satisfying Assumption~2. For almost all realizations $(y_v)$ of $(Y_v)$ it holds that
\begin{equation}\label{eq:clusterconv1}
\begin{aligned}
	& N_n({}\cdot{}) \mid (Y_v)=(y_v) \vd N({}\cdot{};(y_v))
	\qquad\text{and} \\ &
	\tilde N_n({}\cdot{})\mid (Y_v)=(y_v) \vd N({}\cdot{};(y_v))
\end{aligned}
\end{equation}
as $n\to\infty$, where $N({}\cdot{};(y_v))$ is a homogeneous Poisson measure on $C$ with intensity $x^{-\alpha} \eta((y_v))$.

If additionally the field $(Y_v)$ is ergodic, then (unconditionally)
\begin{equation}\label{eq:clusterconv2}
	N_n({}\cdot{}) \vd N({}\cdot{})
	\qquad\text{and} \qquad
	\tilde N_n({}\cdot{})\vd N({}\cdot{})
\end{equation}
as $n\to\infty$, where $N({}\cdot{})$ is a homogeneous Poisson measure on $C$ with intensity $x^{-\alpha} \eta$.
\end{theorem}
In contrast to the cluster measures considered so far, which are defined on $C$ by means of a $\bb c_n$-upscaling, we can naturally also consider corresponding cluster counting measures defined on the original index set scale $\Zd$ as follows:
\begin{equation*}
	L_n(A) = \sum_{z\in \Zd} \I{\{\Jz \cap \Phi_n \cap A \neq \emptyset\}}
\end{equation*}
and
\begin{equation*}
	\tilde L_n(A)
	= \sum_{i=1}^{\Gamma_{n}}\I{\{\CC_i^n \cap A \neq \emptyset\}}
\end{equation*}
for $A \subseteq \Zd$. Using these processes, we can count the number of clusters in more general sequences of sets contained in $D_n = \bb c_n C \cap \Zd$. However, it turns out that we need a slight adjustment of the former geometrical Assumption~\ref{ass:geometry}. With the vector $\bb c_n$ in mind, we define for each $n,k\in\NN$ the side-lengths
\[
	\tilde t_{n,k}^\ell = \frac{c_{n,\ell}}{k^{1/d}}
\]
and collect them in the vector $\bb {\tilde t}_{n,k}$. Similarly as in Section~\ref{sec:mainresults} we then define
\[
	\tilde J_z^{n,k}
	=
	\bb{\tilde t}_{n,k} \bigl(z+[0,1)^d \bigr)	\cap \Zd
\]
for each $z\in\Zd$.

\begin{assumption}\label{ass:geometry2}
Let $(\bb c_n)$, $(D_n)$, %$(\tI)$
and $(\tJ)$ be as above. The sequence $(B_n)_{n\in\NN}$ in $\Zd$ satisfies $B_n\subseteq D_n$ and $\abs{B_n}/\abs{D_n} \to b\ge 0$. Furthermore, if $b> 0$,  
\begin{equation*}%\label{eq:geometricassumption2}
	\limsup_{n\to\infty}\,
	\abs[\big]{\{z\in\Zd\::\: \tJ \subseteq B_n\}}
	\sim b k
	\sim
	\liminf_{n\to\infty}\,
	\abs[\big]{\{z\in\Zd\::\: \tJ \cap B_n \neq \emptyset\}}
\end{equation*} 
as $k\to\infty$.
\end{assumption}

\begin{remark}
As shown in \cite[Lemma~A.2]{RNielsenStehr2022}, the assumption is in particular satisfied if, for all $n$, the set $B_n$ is the lattice-points of a so-called $p$-convex set; see their Definition~2.1. That is, the assumption holds if there is a $p\in\NN$ such that, for all $n$, there exists a set $\overline B_n \subseteq \Rd$ which is a union of (at most) $p$ convex bodies, and which satisfy that $B_n =  \overline B_n\cap \Zd$.
The assumption is also satisfied in the case where $B_n=(\bb b_n B) \cap \Zd$ where $B \subseteq C$ has non-empty interior and boundary of Lebesgue measure zero, and $(\bb b_n)$ is a sequence of vectors satisfying $\bb b_n \to \infty$ and $\bb b_n \le \bb c_n$ where $\lim_{n\to\infty} \bb b_n^*/\bb c_n^*\ge 0$ exists.
\end{remark}

\begin{theorem}\label{thm:originalscale}
Let $C\subseteq\Rd$ and $(D_n)_{n\in\NN}$ be as above, and let $(Y_v Z_v)_\vZ$ be a stochastic volatility field satisfying Assumption~2. 
Let $(B_n^1)_{n\in\NN},\dots,(B_n^G)_{n\in\NN}$ be sequences of sets in $\Zd$, each satisfying Assumption~\ref{ass:geometry2}, such that $B_n^g\subseteq D_n$ for all $n$ and $g$, and $B_n^1,\ldots,B_n^G$ are pairwise disjoint. Assume furthermore that
\[
	\lim_{n\to\infty}\frac{\abs{B_n^g}}{\abs{D_n}}=b_g \in [0,\infty)
\]
for each $g=1,\ldots,G$. Then, for almost all realizations $(y_v)$ of $(Y_v)$,
\begin{equation*}%\label{fml:originalscaleresult}
	\bigl( L_n(B_n^1),\ldots,L_n(B_n^G)\bigr) \mid (Y_v)=(y_v)
	\stackrel{\mathcal{D}}{\to} 
	\bigl(L^1((y_v)),\ldots,L^G((y_v))\bigr)
\end{equation*}
and 
\begin{equation*}%\label{fml:originalscaleresulttilde}
	\bigl(\tilde{L}_n(B_n^1),\ldots,\tilde{L}_n(B_n^G)\bigr) \mid (Y_v)=(y_v)
	\stackrel{\mathcal{D}}{\to}
	\bigl(L^1((y_v)),\ldots,L^G((y_v))\bigr)
\end{equation*}
as $n\to\infty$, where $L^1((y_v)),\ldots,L^G((y_v))$ are independent random variables with each $L^g((y_v))$ being Poisson distributed with parameter $b_g x^{-\alpha}\eta((y_v))$. Here $b_g=0$ means that $\PP(L^g((y_v))=0)=1$. 

If additionally the field $(Y_v)$ is ergodic, then (unconditionally)
\begin{equation*}%\label{fml:originalscaleresults2}
\begin{aligned}
	&
	\bigl( L_n(B_n^1),\ldots,L_n(B_n^G)\bigr)
	\stackrel{\mathcal{D}}{\to} 
	\bigl(L^1,\ldots,L^G\bigr) \qquad\text{and}\\&
	\bigl(\tilde{L}_n(B_n^1),\ldots,\tilde{L}_n(B_n^G)\bigr)
	\stackrel{\mathcal{D}}{\to} 
	\bigl(L^1,\ldots,L^G\bigr)
\end{aligned}
\end{equation*}
as $n\to\infty$, where $L^1,\ldots,L^G$ are independent, Poisson distributed random variables with parameter $b_g x^{-\alpha}\eta$.
\end{theorem}

%%%%%%%%%%%%%%%%%%%%%%%%
\section{Proof of results from Section~\ref{sec:mainresults}}
\label{sec:proofmaxresults}
%%%%%%%%%%%%%%%%%%%%%%%%e

In this section we present the proofs of Theorem~\ref{thm:maxtail} and  Corollaries~\ref{cor:maxtail2}--\ref{cor:Zisstrongmixing}, and thus the underlying assumptions are assumed satisfied throughout the section. That is, the index sets $(D_n)$ satisfy Assumption~\ref{ass:geometry} and the fields $(Z_v)$, $(Y_v)$ and $(Z_v^t)$, $t\in\NN$, satisfy the conditions specified in Assumption~\ref{ass:randomfield}. The results are proven by approximating the extremal behavior of $(y_v Z_v)$ with that of $(y_v Z_v^t)$ for a realization $(y_v)$ of $(Y_v)$, which, as seen in Lemma~\ref{lem:extremalapprox1}, is made possible due to the following lemma.

\begin{lemma}\label{lem:ratiobound}
Let $\gamma>\alpha$ be given. Then there exist $c>0$ and $x_0>0$ such that
\begin{equation}\label{eq:ratiobound}
	\frac{\PP( y Z_0>x)}{\PP(Z_0>x)} \le 
	c\, \bigl( \abs{y}^\gamma + \I{\{\abs{y}\le 1\}} \bigr)
\end{equation}
for all $x\ge x_0$ and all $y\in \R$. Additionally, for all $t\in\NN$ there exists $c_t>0$ such that
\begin{equation}\label{eq:ratioboundwitht}
	\frac{\PP( y Z_0^t>x)}{\PP(Z_0>x)} \le 
	c_t \bigl( \abs{y}^\gamma + \I{\{\abs{y}\le 1\}} \bigr)
\end{equation}
for $x\ge x_0$ and all $y\in \R$. Moreover, for all $t\in\NN$ there exist $K_t$ satisfying $K_t \downarrow 0$ and $x_t>0$ such that
\begin{equation}\label{eq:ratiobound2}
	\frac{\PP(y Z_0^t> x, y Z_0 \le x)}{\PP(Z_0>x)} \vee
	\frac{\PP(y Z_0> x, y Z_0^t \le x)}{\PP(Z_0>x)}
	\le K_t \bigl(\abs{y}^\gamma + \I{\{\abs{y}\le 1\}} \bigr)
\end{equation}
for all $x\ge x_t$ and $\abs{y} \le x/ x_t$.
\end{lemma}

\begin{proof}
The inequality \eqref{eq:ratiobound} follows by Karamata's representation theorem \cite[Theorem~A3.3]{Embrecths1997} for regularly varying functions combined with the tail balance condition \eqref{eq:tailbalance}; see for instance \cite[Eq.~(A.1)]{RNielsenStehr2022} for a similar bound. The inequality \eqref{eq:ratioboundwitht} follows by combining \eqref{eq:ratiobound} with \eqref{eq:tailbalance} and \eqref{eq:approx1}.

For the inequality \eqref{eq:ratiobound2} we first note that \eqref{eq:C+C-} and \eqref{eq:approx2} from Assumption~\ref{ass:randomfield} imply
\begin{align*}
	& \,\limsup_{t\to\infty}\limsup_{x\to\infty} 
	\frac{
	\PP( Z_0^t>x, Z_0 \le x)}{\PP(Z_0>x)}
	=\limsup_{t\to\infty}\limsup_{x\to\infty} 
	 \frac{
	\PP( Z_0^t<-x, Z_0 \ge -x)}{\PP(Z_0<-x)}
	 \\ 
	= &\, 	
	\limsup_{t\to\infty}\limsup_{x\to\infty} 
	 \frac{
	\PP( Z_0>x, Z_0 \le x)}{\PP(Z_0>x)}
	=	
	\limsup_{t\to\infty}\limsup_{x\to\infty} 
	 \frac{
	\PP( Z_0<-x, Z_0^t \ge -x)}{\PP(Z_0<-x)} \\
	= &\, 0 ,
\end{align*}
which is equivalent to the existence of $0<x_t \uparrow \infty$ and $\tilde K_t \downarrow 0$ such that all the fractions above are bounded by $\tilde K_t$ for any $x\ge x_t$. Let $t$ and consequently $\tilde K_t$ and $x_t$ be fixed. We show the inequality \eqref{eq:ratiobound2} for the first fraction of the left hand side only, as the bound for the second fraction follows identically.

The result is trivially true for $y=0$. Let $y\neq 0$, $x\ge x_t$ and write
\[
	\frac{\PP(y Z_0^t> x, y Z_0 \le x)}{\PP(Z_0>x)}
	=
	\frac{\PP(y Z_0^t> x, y Z_0 \le x)}{\PP(y Z_0>x)}
	\frac{\PP(yZ_0>x)}{\PP(Z_0>x)} .
\]
The first factor is bounded by $\tilde K_t$ for all $\abs{y} \le x/x_t$ by assumption, and the second factor is bounded by $c\, (\abs{y}^\gamma + \I{\{\abs{y}\le 1\}})$ for all such $y$, since we may assume that $x_t \ge x_0$. Defining $K_t=\tilde K_t \,c \downarrow 0$ yields the claim.
\end{proof}

\begin{lemma}\label{lem:extremalapprox1}
For almost all realizations $(y_v)$ of $(Y_v)$,
\begin{equation}\label{eq:extremalapprox1}
\limsup_{t\to\infty}\limsup_{n\to\infty}\abs[\Big]{
	\PP \bigl( \max_{v\in D_n} y_v Z_v \le a_n x \bigr) -
	\PP \bigl( \max_{v\in D_n} y_v Z_v^{t} \le a_n x \bigr)
	}
	= 0 
\end{equation}
for any $x>0$.
\end{lemma} 

\begin{proof}
Note that the difference in \eqref{eq:extremalapprox1} is bounded by
\begin{align}
	\nonumber
	& \sum_{v\in D_n} \PP\bigl(y_v Z_v^t>a_n x,\, y_v Z_v \le a_n x \bigr)
	+ \sum_{v\in D_n} \PP\bigl(y_v Z_v>a_n x,\, y_v Z_v^t \le a_n x \bigr)\\ &
	\label{eq:differencebound}	
	=\, 
	\sum_{v\in D_n} \PP\bigl(y_v Z_0^t>a_n x,\, y_v Z_0 \le a_n x\bigr)
	+ \sum_{v\in D_n} \PP\bigl(y_v Z_0>a_n x,\, y_v Z_0^t \le a_n x\bigr) ,
\end{align}
and thus, to show \eqref{eq:extremalapprox1}, it suffices to show that  $\limsup$ (in $n$ and $t$) of each sum equals zero. We only consider the first sum in \eqref{eq:differencebound}, as the second sum is handled identically.

Fix $x>0$ and $\gamma>\alpha$ such that \eqref{eq:gammaY} is satisfied. For all $t$ let $x_t>0$ and $K_t$ from Lemma~\ref{lem:ratiobound} be given. Choose $n$ large enough such that $a_n x > x_t$. Then we obtain from Lemma~\ref{lem:ratiobound} eq. \eqref{eq:ratiobound2} that
\begin{align*}
	\MoveEqLeft	
	\sum_{v\in D_n} \PP(y_v Z_0^t>a_n x,\, y_v Z_0 \le a_n x) \I{\{\abs{y_v} \le a_n x / x_t\}} \\ &
	\le \frac{\abs{D_n}\PP(Z_0 > a_n x)}{\abs{D_n}}
	\sum_{v\in D_n} K_t \bigl(\abs{y_v}^\gamma+1 \bigr) \I{\{\abs{y_v} \le a_n x / x_t\}}  \\ &
	\le \frac{\abs{D_n}\PP(Z_0 > a_n x)}{\abs{D_n}}
	\sum_{v\in D_n} K_t \bigl(\abs{y_v}^\gamma+1 \bigr) \\ &
	\to x^{-\alpha} K_t\, \EE^{\bb{Y}}[\abs{\hat{Y}_0}^\gamma + 1\mid \hat{\Ii}]((y_v))
\end{align*}
as $n\to\infty$, 
with the convergence justified by Lemma~\ref{lem:ergodictheorem} of Appendix~\ref{app:ergodictheorem}. Secondly letting $t\to\infty$ we see that
\[
	\limsup_{t\to\infty} \limsup_{n\to\infty}
	\sum_{v\in D_n} \PP(y_v Z_0^t>a_n x,\, y_v Z_0 \le a_n x) \I{\{\abs{y_v} \le a_n x / x_t\}} = 0.
\]

Now let $b\ge 1$ be given, and choose $n$ such that $a_n x / x_t \ge b$. Then from \eqref{eq:ratioboundwitht} we
find that
\begin{align*}
	\MoveEqLeft	
	\sum_{v\in D_n} \PP\bigl( y_v Z_0^t>a_n x,\, y_v Z_0 \le a_n x \bigr) \I{\{\abs{y_v} > a_n x / x_t\}} \\ &
	\le \frac{\abs{D_n}\PP(Z_0>a_n x)}{\abs{D_n}}
	\sum_{v\in D_n} \frac{\PP(y_v Z_0^t > a_n x)}{\PP( Z_0 >a_n x)}\I{\{\abs{y_v} > b\}} \\ &
	\le \frac{\abs{D_n}\PP(Z_0>a_n x)}{\abs{D_n}}
	\sum_{v\in D_n} c_t\, \abs{y_v}^\gamma \I{\{\abs{y_v} > b\}} \\ &
	\to x^{-\alpha} c_t\, \EE^{\bb{Y}}[\abs{\hat{Y}_0}^\gamma \I{\{\abs{\hat{Y}_0}>b\}}\mid \hat{\Ii}]((y_v))
\end{align*}
as $n\to\infty$, again justified by Lemma~\ref{lem:ergodictheorem} of Appendix~\ref{app:ergodictheorem}. Letting $b\to\infty$ shows that
\[
	\limsup_{n\to\infty}
	\sum_{v\in D_n} \PP(y_v Z_0^t>a_n x,\, y_v Z_0 \le a_n x) \I{\{\abs{y_v} > a_n x / x_t\}} = 0,
\]
and \eqref{eq:extremalapprox1} follows.
\end{proof}

Next, we turn to finding a limit of $\PP \bigl( \max_{v\in D_n} y_v Z_v^{t} \le a_n x \bigr)$ as $n\to\infty$ leading up to the proof of Theorem~\ref{thm:maxtail} below. For this, we first show a series of lemmas which utilize certain technical results found in the appendix of the paper.

The following lemma is a consequence of the assumption that $(Z^t_v)$ satisfies the anti-clustering condition \ref{eq:fieldass3} of Assumption~\ref{ass:randomfield}. In this and subsequent lemmas the following notation will be useful: For each $K\in\NN$ we define $f^K,\overline{f}^K\,:\,\R\to\R$ by
\[
	f^K(y)=y\I{\{\abs{y}\leq K\}} 
	\qquad\text{and} \qquad 
	\overline{f}^K(y)=y\I{\{\abs{y}> K\}} .
\]
Moreover, here and in the remainder of the paper (including the appendix), we shorten $\max_{v\in A} R_v$ by $M_R(A)$ for any $A\subseteq \Zd$, where $(R_v)$ is a general field. 

\begin{lemma}
\label{lem:condanticluster}
Let $t\in\NN$ and $K\in\NN$. It holds for all deterministic fields $(y_v)$ that
\[
	\lim_{m\to\infty}\limsup_{n\to\infty}
	\sum_{v\in D_n}
	\PP\bigl(M_{\abs{f^K(y)Z^t}}(\Av\setminus \Avm)>a_n x,\,
	\abs{f^K(y_v)Z^t_v}>a_n x \bigr)
	=0
\]
for any $x>0$.
\end{lemma}

\begin{proof}
Write
\begin{align*}
	\MoveEqLeft
	\sum_{v\in D_n}
	\PP \bigl(M_{\abs{f^K(y)Z^t}}(\Av\setminus\Avm)>a_n x,\, \abs{f^K(y_v)Z^t_v}>a_n x \bigr)\\ &
	\le \sum_{v\in D_n}
	\PP\bigl(M_{\abs{Z^t}}(\Av\setminus\Avm)>a_n x/K,\, 
	\abs{Z^t_v}>a_n x/K \bigr)\\ &
	= \abs{D_n}\,
	\PP \bigl(M_{\abs{Z^t}}(\An\setminus\Anm)>a_nx/K \mid \abs{Z^t_0}>a_nx/K\bigr)\,
	\PP\bigl(\abs{Z^t_0}>a_n x/K \bigr).
\end{align*}
Now the result follows from Assumption~\ref{ass:randomfield}\ref{eq:fieldass3}, since 
\[
	\lim_{n\to\infty}
	\abs{D_n} \PP \bigl(\abs{Z^t_0}>a_n x/K \bigr)
	= K^{\alpha}x^{-\alpha} \Bigl(C_t^++C_t^-\frac{1-p}{p} \Bigr)
	< \infty
\]
due to \eqref{eq:tailbalance}, \eqref{eq:approx1} and the choice of $(a_n)$.
\end{proof}

\begin{lemma}\label{lem:Hsingconsequence}
Let $t\in\NN$ and $K\in\NN$. For almost all realizations $(y_v)$ of $(Y_v)$ it holds that
\begin{equation}\label{eq:mlimit}
\begin{aligned}
	\MoveEqLeft
	\lim_{m\to\infty}\lim_{n\to\infty} 
	\sum_{v\in D_n}
	\PP \bigl(M_{f^K(y)Z^t}(\Avm)\le a_n x < f^K(y_v)Z^t_v \bigr)\\&
	=
	\lim_{n\to\infty}
	\sum_{v\in D_n}
	\PP \bigl(M_{f^K(y)Z^t}(\Av)\le a_n x < f^K(y_v)Z^t_v \bigr)
\end{aligned}
\end{equation}
for all $x>0$, where all limits exist.
\end{lemma}

\begin{proof}
Since $(Z_v^t)$ is regularly varying, we have from Lemma~\ref{lem:extremalindexgivenyfixedm} in Appendix~\ref{app:regvar} that
\[
	\lim_{n\to\infty}
	\sum_{v\in D_n}
	\PP\bigl(M_{f^K(y)Z^t}(\Avm)\le a_n x < f^K(y_v) Z^t_v \bigr)
\]
exists for each $m\in\NN$. Furthermore, the limit is decreasing and non-negative as $m\to\infty$, so also
\[
	\lim_{m\to\infty}\lim_{n\to\infty}\sum_{v\in D_n}
	\PP \bigl(M_{f^K(y)Z^t}(\Avm) \le a_n x <f^K(y_v) Z^t_v \bigr)
\]
exists. Additionally,
\begin{align*}
	0	&
	\le \sum_{v\in D_n}
	\PP \bigl(M_{f^K(y)Z^t}(\Avm)\le a_n x < f^K(y_v)Z^t_v \bigr)\\ &
	\qquad
	-\sum_{v\in D_n}
	\PP\bigl(M_{f^K(y)Z^t}(\Av)\le a_n x <f^K(y_v)Z^t_v \bigr)\\ &
	\le \sum_{v\in D_n}
	\PP \bigl (M_{\abs{f^K(y)Z^t}}(\Av\setminus \Avm)>a_n x,\, \abs{f^K(y_v)Z^t_v}>a_n x \bigr)
\end{align*}
vanishes in the limit $\lim_m\lim_n$ by Lemma~\ref{lem:condanticluster}. This concludes the proof.
\end{proof}

\begin{lemma}\label{lem:maxtailZt}
Let $t\in\NN$. For almost all realizations $(y_v)$ of $(Y_v)$ it holds that
\begin{equation}\label{eq:maxtailZt}
	\lim_{n\to\infty}
	\PP\bigl(
	\max_{v\in D_n} y_v Z^t_v \le a_n x	
	 \bigr)
	= \exp \Bigl( -x^{-\alpha} C_t^+\lim_{K\to\infty} \lim_{m\to\infty}\eta^{t,K,m}((y_v)) \Bigr)
\end{equation}
for all $x>0$ with $\eta^{t,K,m}$ given in Theorem~\ref{thm:maxtail}.
\end{lemma}

\begin{proof}
For each $K\in\NN$ we have that
\begin{equation}\label{eq:upperbound}
	\PP \bigl(\max_{v \in D_n} y_vZ^t_v \le a_n x \bigr) 
	\le 
	\PP \bigl(\max_{v \in D_n} f^K(y_v)Z^t_v \le a_n x \bigr)
\end{equation}
and
\begin{equation}\label{eq:lowerbound}
\begin{aligned}
	\MoveEqLeft
	\PP\bigl(\max_{v \in D_n} y_vZ^t_v \le a_n x \bigr)\\ &
	=
	\PP\bigl(\max_{v \in D_n} f^K(y_v)Z^t_v \le a_n x,\,
	\max_{v \in D_n} \overline{f}^K(y_v)Z^t_v \le a_n x \bigr)\\&
	\ge 
	\PP\bigl(\max_{v \in D_n} f^K(y_v)Z^t_v \le a_n x\bigr)-
	\PP\bigl(\max_{v \in D_n} \overline{f}^K(y_v)Z^t_v > a_n x \bigr).
\end{aligned}
\end{equation}
For the second term in the right hand side of \eqref{eq:lowerbound} we have from Lemma~\ref{lem:maxvsJkasse}~eq.~\eqref{eq:maxprobinequality} of Appendix~\ref{app:Geometry} that
\begin{equation}\label{eq:lowerboundoffbar}
\begin{aligned}
	\PP\bigl(\max_{v \in D_n} \overline{f}^K(y_v)Z^t_v \le a_n x \bigr)&
	\ge \exp\Bigl(-\sum_{v\in D_n}\PP ( \overline{f}^K(y_v)Z^t_v> a_n x ) \Bigr) + o(1)\\&
	\ge \exp\Bigl(-x^{-\alpha}\,c_t\,\EE^{\bb{Y}}(\abs{\hat{Y}_0}^\gamma\I{\{\abs{\hat{Y}_0}>K\}}\mid \hat{\Ii}) \Bigr) + o(1)
\end{aligned},
\end{equation}
where the last inequality follows from Lemma~\ref{lem:ratiobound}, the choice of $(a_n)$ and Lemma~\ref{lem:ergodictheorem} of Appendix~\ref{app:ergodictheorem}. From the assumption that $\EE\abs{Y_0}^\gamma<\infty$, we clearly have that the lower bound in \eqref{eq:lowerboundoffbar} converges to $1$, when $K$ tends to infinity.

Applying Lemma~\ref{lem:JsetandAset} from Appendix~\ref{app:Geometry} to the field $(f^K(y_v)Z^t_v)_{\vZ}$ we find that
\begin{equation}\label{eq:fMlimit}
	\begin{aligned}
	\MoveEqLeft
	\PP\bigl(\max_{v \in D_n} f^K(y_v)Z^t_v \le a_n x \bigr)\\ & 
	=
	\exp\Bigl(-\sum_{v\in D_n}\PP \bigl( M_{f^K(y)Z^t}(\Av) \le a_n x
	< f^K(y_v)Z^t_v \bigr) \Bigr) + o(1)
\end{aligned}
\end{equation}
as $n\to\infty$, where the sum has a well-defined limit given by \eqref{eq:mlimit}.
Since the probability $\PP(\max_{v \in D_n} f^K(y_v)Z^t_v \le a_n x )$ is decreasing in $K$, the (right hand side) limit in \eqref{eq:mlimit} furthermore has a well-defined limit as $K$ tends to infinity. Additionally, from Lemma~\ref{lem:extremalindexgivenyfixedm} of Appendix~\ref{app:regvar} we have
\begin{equation}\label{eq:regvarlimit}
\begin{aligned}
	&\lim_{n\to\infty}\sum_{v\in D_n}
	\PP \bigl( M_{f^K(y)Z^t}(A_v^{(m)}) \le a_n x < f^K(y_v)Z^t_v \bigr)\\
	&=x^{-\alpha}\,C_t^+
	\frac{\EE^{\bb{Y}}\Bigl[\EE\Bigl(\bigl((f^K(Y_0)\Theta^{t,m}_0)_+^\alpha-(\max_{v\in A^{(m)}_0}f^K(Y_v)\Theta^{t,m}_v)_+^\alpha\bigr)_+\mid\bb{Y}=\hat{\bb{Y}}\Bigr)\mid\hat{\Ii}\Bigr]((y_v))}{\EE(\Theta_0^{t,m})_+^\alpha}.
\end{aligned}
\end{equation}
Letting $K$ tend to infinity in \eqref{eq:upperbound} and \eqref{eq:lowerbound} and combining with \eqref{eq:mlimit}, \eqref{eq:lowerboundoffbar}--\eqref{eq:regvarlimit} gives the desired result.
\end{proof}

\begin{proof}[Proof of Theorem~\ref{thm:maxtail}]
From Lemma~\ref{lem:maxtailZt} we have that $\PP ( \max_{v\in D_n} y_v Z_v^t \le a_n x)$ has a well-defined limit as $n\to\infty$ for all $t\in\NN$.  Combining with Lemma~\ref{lem:extremalapprox1} we obtain that the limit of 
$\PP ( \max_{v\in D_n} y_v Z_v \le a_n x)$ exists and is given by
\[
	\lim_{n\to \infty}
	\PP \bigl( \max_{v\in D_n} y_v Z_v \le a_nx \bigr)
	= \lim_{t\to\infty}\lim_{n\to \infty}
	\PP \bigl( \max_{v\in D_n} y_v Z_v^t \le a_nx \bigr).
\]
The statement of the theorem now follows from \eqref{eq:maxtailZt} since $C_t^+$ tends to $1$ by assumption.
\end{proof}

\begin{proof}[Proof of Corollary~\ref{cor:maxtail2}]
The result is an immediate consequence of Theorem~\ref{thm:maxtail} and  Lemmas~\ref{lem:tailapprox} and \ref{lem:JsetandAset} of Appendix~\ref{app:Geometry}.
\end{proof}

\begin{proof}[Proof of Corollary~\ref{cor:maxtailergodic}]
Assuming that $(Y_v)$ is ergodic makes the invariant $\sigma$-algebra $\hat{\Ii}$ trivial, and the first part of the corollary follows directly.  Taking expectations in (\ref{eq:main1}) and using dominated convergence then gives
\[
	\lim_{n\to\infty}
	\PP\bigl( \max_{v\in D_n} Y_v Z_v \le a_n x\bigr) = \exp ( -x^{-\alpha}\eta).
\]
The statement on the extremal index of $(Y_vZ_v)_{v\in\Zd}$ follows from combining this with the convergence (\ref{eq:tailofY0Z0}) concerning the marginal distribution.
\end{proof}

\begin{proof}[Proof of Corollary~\ref{cor:mtdpendence}]
Let $t\in\NN$ be fixed, and choose $m\ge m_t$. Then the random variables $M_{yZ^t}(\Av\setminus \Avm)$ and $y_v Z_v^t$ are independent for any $v\in D_n$. Consequently,
\begin{align*}
	\MoveEqLeft	
	\sum_{v\in D_n}
	\PP \bigl( M_{yZ^t}(\Av \setminus \Avm) > a_nx,\,  y_v Z_v^t > a_n x\bigr) \\&
	= \sum_{v\in D_n}
	\PP \bigl( M_{yZ^t}(\Av \setminus \Avm) > a_nx\bigr) \,
	\PP\bigl(y_v Z_v^t > a_n x\bigr) \\ &
	\le \sum_{v\in D_n} 
	\PP\bigl(y_v Z_v^t > a_n x\bigr)
	\sum_{u\in \Av}
	\PP \bigl( y_u Z_u^t > a_n x\bigr) \\ &
	\to 0
\end{align*}
as $n\to\infty$. The convergence, which holds for almost all realizations $(y_v)$ of $(Y_v)$, can be obtained using arguments similar to those of Appendix~\ref{app:Geometry} utilizing the convergence result of Appendix~\ref{app:ergodictheorem}. 
Having obtained this convergence, going through the arguments of the proofs of Lemmas~\ref{lem:Hsingconsequence} and \ref{lem:maxtailZt} without the $K$-restriction (i.e. with $K=\infty$) shows that  
\[
	\lim_{n\to\infty}
	\PP\bigl(
	\max_{v\in D_n} y_v Z^t_v \le a_n x	
	 \bigr)
	= \exp \Bigl( -x^{-\alpha} \lim_{m\to\infty}\eta^{t,m}((y_v)) \Bigr),
\]
with $\eta^{t,m}$ given in the present corollary. The result now follows by an application of Lemma~\ref{lem:extremalapprox1}.
\end{proof}

\begin{proof}[Proof of Corollary~\ref{cor:Zisstrongmixing}]
This is a trivial consequence since all fields $(Z^t_v)$ can be chosen identical to $(Z_v)$. 
\end{proof}

%%%%%%%%%%%%%%%%%%%%%%%%
\section{Proof of results from Section~\ref{sec:mainresultscluster}}
\label{sec:proofclusterprocess}
%%%%%%%%%%%%%%%%%%%%%%%%

In this section we present the proofs of Theorems~\ref{thm:clusterconvergence} and \ref{thm:originalscale}, and therefore the underlying assumptions hold for now. In particular we consider a bounded set $C\subseteq \Rd$ with non-empty interior and a boundary of volume zero. To ease notation, we define
\begin{align*}
	P_n(A) & = \bigl\{
		z\in\Zd\::\: \Jz \subseteq \bb c_n A 
	\bigr\} \qquad \text{and}\\ 
	Q_n(A) & = \bigl\{
		z\in\Zd\::\: \Jz \cap \bb c_n A \neq \emptyset
	\bigr\}
\end{align*}
for each $A\in\Bb(C)$.

\begin{proof}[Proof of Theorem~\ref{thm:clusterconvergence}]
We only show the conditional convergences of \eqref{eq:clusterconv1} as the unconditional ones of \eqref{eq:clusterconv2} follow immediately.

We start by showing the vague convergence $N_n \vd N$ conditioned on $(Y_v)=(y_v)$. To ease notation, we write $N_n({}\cdot{};(y_v))$ as the conditional cluster measure. According to \cite[Theorem~4.18]{Kallenberg2017} it is enough to show that $\EE N_n(A;(y_v)) \to \EE N(A;(y_v))$ for any half-open box $A\subseteq C$, and that $\PP(N_n(B;(y_v))=0) \to \PP(N(B;(y_v))=0)$ for any union $B\subseteq C$ of such boxes.

Let $A = \bigtimes_{\ell=1}^d (a_\ell,b_\ell] \subseteq C$, where $a_\ell < b_\ell$ for all $\ell=1,\dots,d$. Indeed, $A$ is a bounded set with non-empty interior and boundary of volume zero, and therefore the sequence of sets $(\bb c_n A)\cap \Zd$ satisfy Assumption~\ref{ass:geometry} with $\Jz$-boxes with side-lengths given by the vector $\bb t_n$. Thus, applying the results of the previous sections to this set, we obtain (similarly to Corollary~\ref{cor:maxtail2}) that
\begin{equation}\label{eq:sumAconv}
	\lim_{n\to\infty} \sum_{z\in P_n(A)} 
	\PP \bigl( M_{yZ}(\Jz)>a_n x\bigr)
	= \abs{A} x^{-\alpha} \eta((y_v)) ,
\end{equation}
and similarly with $P_n(A)$ replaced by $Q_n(A)$.
Be definition of $N_n$ we therefore see that
\begin{align*}
	\EE N_n(A;(y_v)) &
%	= \EE \Bigl[
%	\sum_{z\in\Zd} 
%	\I{\{\Jz \cap \Phi_n \cap \bb c_n A \neq \emptyset\}} \Bigr]\\ &
	\sim \sum_{z\in P_n(A)} 
	\PP \bigl( M_{yZ}(\Jz)>a_n x\bigr) \\ &
	\to \abs{A} x^{-\alpha}\eta((y_v)) \\ &
	= \EE N(A;(y_v)) 
\end{align*}
as $n\to\infty$,
where the asymptotic equivalence is due to \eqref{eq:sumAconv} and the fact that
\[
	\sum_{z\in P_n(A)} \I{\{M_{yZ}(\Jz)>a_n x\}}
	\le N_n(A;(y_v))
	\le \sum_{z\in Q_n(A)} \I{\{M_{yZ}(\Jz)>a_n x\}}
\]
for all $n$. 

Now let $B\subseteq C$ be a finite union of boxes, i.e.
\[
	B = \bigcup_{i=1}^r \Bigl( 
	\bigtimes_{\ell=1}^d (a^i_\ell,b^i_\ell]
	\Bigr)
\]
for some $r\in\NN$ and some $a^i_\ell < b^i_\ell$ for $i=1,\dots,r$ and $\ell=1,\dots,d$. Again, as the sequence of sets $(\bb c_n B) \cap \Zd$ satisfies Assumption~\ref{ass:geometry} based on the same $\Jz$-boxes, utilizing the results of previous sections we get (similarly to Theorem~\ref{thm:maxtail})
\begin{align*}
	\PP\bigl(N_n(B;(y_v)) = 0 \bigr) &
	\sim \PP\Bigl( M_{yZ} \bigl(\bigcup_{z\in P_n(B)} \Jz \bigr) \le a_n x\Bigr) \\ &
	\to \exp \Bigl( -\abs{B} x^{-\alpha} \eta((y_v))\Bigr)\\&
	= \PP\bigl(N(B;(y_v))=0 \bigr)
\end{align*}
as $n\to\infty$. Hence, $N_n({}\cdot{};(y_v)) \vd N({}\cdot{};(y_v))$.

To show that also (with similar notation as above) $\tilde N_n({}\cdot{};(y_v)) \vd N({}\cdot{};(y_v))$, it again suffices to show that $\EE \tilde N_n(A;(y_v)) \to \EE N(A;(y_v))$ and $\PP(\tilde N_n(B;(y_v))=0) \to \PP(N(B;(y_v))=0)$ for sets $A,B\subseteq C$ as above. As the proof follows similar arguments as the proof of \cite[Theorem~5.2]{RNielsenStehr2022}, we simply sketch the structure here and refer to their proof for details: By the definition of the cluster processes, a cluster counted by $\tilde N_n$ will be counted by at least one cluster by $N_n$ as well. However, for $m\in\NN$, there exists a set $E_{m,n}$ with probability tending to $1$ in the double limit $\lim_m\lim_n$, such that a cluster counted by $\tilde N_n$ corresponds to at most one cluster counted by $N_n$, and thus on this set the cluster measures are identical. The two sufficient convergences for $\tilde N_n$ are then obtained by the corresponding convergences for $N_n$ by restricting to the set $E_{m,n}$ and then letting first $n\to\infty$ and then $m\to\infty$.
\end{proof}

\begin{proof}[Proof of Theorem~\ref{thm:originalscale}]
The results follow by similar considerations as Theorem~7.1 of \cite{RNielsenStehr2022}.
\end{proof}

%%%%%%%%%%%%%%%%%%%%%%%%
\section{Example~\ref{ex:movingaverage}: Moving average}
\label{sec:movingaverage}
%%%%%%%%%%%%%%%%%%%%%%%%

Recall that we consider the moving average field $(Z_v)_\vZ$ given by
\[
	Z_v = \sum_{u\in\Zd} \psi_u \xi_{v-u} ,
\]
where the random variables $(\xi_v)$ are i.i.d. with regularly varying right tail of index $\alpha$ satisfying the tail-balance condition $(1-p_\xi) \PP(\xi_0>x) \sim p_\xi \PP(\xi_0< -x)$ as $x\to\infty$. It is assumed that 
\[
	\sum_{u \in \Zd} \abs{\psi_u}^{\delta} < \infty	
\]
for some $\delta \in (0,\alpha) \cap (0,1]$, which ensures the absolute almost sure convergence of $Z_v$. Defining the approximating field $(Z_v^t)_\vZ$ by
\[
	Z_v^t = \sum_{\norm{u}\le t} \psi_u \xi_{v-u} ,
\]
we show in this section that $(Z_v)$ fulfills Assumption~\ref{ass:randomfield}, and we explicitly find the relevant spectral fields and corresponding extremal functional $\eta(\cdot)$. 

In \cite[Section~3]{Jakubowski2019} it is shown that a similar moving average field is $\max$-$t$-approximable, which in turn means that  \ref{eq:fieldass2}--\ref{eq:fieldass5} of Assumption~\ref{ass:randomfield} are satisfied. However, the concept of $\max$-$t$-approximability supplies a different sequence of approximating fields, $(\tilde{Z}^t_v)$, where each $\tilde{Z}_v^t$ is not defined from only finitely many $\xi_v$-variables. Such fields are less suitable for our purpose, where we find spectral distributions via finite-dimensional distributions.

\begin{lemma}\label{lem:MA1}
The moving average field $(Z_v)$ in combination with its approximating fields $(Z_v^t)$ fulfill \eqref{eq:tailbalance} and \ref{eq:fieldass2}--\ref{eq:fieldass5} of Assumption~\ref{ass:randomfield}. The tail-balance \eqref{eq:tailbalance} is satisfied with $p=\sum \bigl(p_\xi (\psi_u)_+^\alpha + (1-p_\xi) (\psi_u)_-^\alpha\bigr)/\norm{\psi}_\alpha^\alpha$, where $\norm{\psi}_\alpha^\alpha$ is defined in (\ref{eq:normpsi}).
\end{lemma}

\begin{proof}
For any $t\in\NN$ clearly $(Z_v^t)_\vZ$ is $2t$-dependent and thus \ref{eq:fieldass2} and \ref{eq:fieldass3} are trivially satisfied.

For any $t$ (including $t=\infty$ giving $Z^t=Z$) we find from a spatial version of \cite[Lemma~A3.26]{Embrecths1997} that
\begin{equation}\label{eq:MAtaileq}
	\frac{\PP(Z_0^t>x)}{\PP(\abs{\xi_0} > x)}
	\to
	\sum_{\norm{u}\le t} \Bigl( 
	p_\xi (\psi_u)_+^\alpha + (1-p_\xi) (\psi_u)_-^\alpha 	
	\Bigr) 
\end{equation}
as $x\to\infty$. Realizing that the fraction $\PP(Z_0^t<-x)/\PP(\abs{\xi_0}>x)$ has a similar limit simply replacing $p_\xi$ with $(1-p_\xi)$, we obtain that
\[
	\frac{\PP(Z_0<-x)}{\PP(Z_0>x)}
	\to
	\frac{1-\sum_u \bigl(p_\xi (\psi_u)_+^\alpha + (1-p_\xi) (\psi_u)_-^\alpha\bigr)/\norm{\psi}_\alpha^\alpha}{\sum_u \bigl(p_\xi (\psi_u)_+^\alpha + (1-p_\xi) (\psi_u)_-^\alpha\bigr)/\norm{\psi}_\alpha^\alpha} .
\]
This shows \eqref{eq:tailbalance} with the given $p$. The asymptotic equivalence between $Z_v^t$ and $Z_v$ of \ref{eq:fieldass4} follows immediately since
\[
	\lim_{x\to\infty}
	\frac{\PP(Z_0^t>x)}{\PP(Z_0>x)}
	= \frac{\sum_{\norm{u}\le t} \bigl(p_\xi (\psi_u)_+^\alpha + (1-p_\xi) (\psi_u)_-^\alpha \bigr)}{\sum_u \bigl(p_\xi (\psi_u)_+^\alpha + (1-p_\xi) (\psi_u)_-^\alpha\bigr)}
	\to 1
\]
as $t \to \infty$, with a similar expression for the fraction of left tails replacing $p_\xi$ with $(1-p_\xi)$.

We now set out to show \ref{eq:fieldass5} of Assumption~\ref{ass:randomfield}. Let $d_t \downarrow 0$ at a rate determined soon. For such a $d_t$ it is easily seen that
\begin{align*}
	\PP\bigl(Z_0 > (1+d_t) x \bigr) - \PP\bigl( \abs{Z_0 - Z_0^t}>d_t x\bigr) &
	\le \PP\bigl(Z_0 > (1+d_t) x,\,  \abs{Z_0 - Z_0^t} \le d_t x\bigr) \\ &
	\le \PP\bigl(Z_0^t > x,\, Z_0>x \bigr) .
\end{align*}
As a consequence of \eqref{eq:MAtaileq} using $t=\infty$ we in particular have that the right tail of $Z_0$ is regularly varying with index $\alpha$, and consequently 
\[
	\lim_{x\to\infty} \frac{\PP(Z_0>(1+d_t)x)}{\PP(Z_0>x)}
	= (1+d_t)^{-\alpha}
	\to 1
\]
as $t\to\infty$. Thus we have
\[
	\liminf_{t\to\infty}\liminf_{x\to\infty} 
	\frac{
	\PP( Z_0^t>x, Z_0 > x)}{\PP(Z_0>x)} %\\ &
	=
	1
\]
once it is shown that
\[
	\limsup_{t\to\infty} \limsup_{x\to\infty}
	\frac{\PP(\abs{Z_0-Z_0^t}>d_t x)}{\PP(Z_0>x)} = 0 .
\]
Again utilizing \cite[Lemma~A3.26]{Embrecths1997} it can be seen that
\[
	\lim_{x\to\infty} 
	\frac{\PP(\abs{Z_0-Z_0^t}>d_t x)}{\PP(Z_0>x)}
	= d_t^{-\alpha} \frac{\sum_{\norm{u}> t} \abs{\psi_u}^\alpha}{\sum_u \bigl(p_\xi (\psi_u)_+^\alpha + (1-p_\xi) (\psi_u)_-^\alpha\bigr)} ,
\]
which tends to 0 e.g. when defining $d_t$ by 
\[
	d_t = \bigl( \sum_{\norm{u}>t} \abs{\psi_u}^\alpha \bigr)^{1/(2\alpha)} .
\]
It can be shown in the same manner that also
\[
	\liminf_{t\to\infty}\liminf_{x\to\infty} 
	 \frac{
	\PP( Z_0^t<-x, Z_0<-x)}{\PP(Z_0<-x)} %\\ &
	= 1,
\]
which concludes the proof.
\end{proof}

\begin{lemma}
Let $(Z_v)$ and $(Z_v^t)Z$ be as above, and assume that $(Y_v)_{v\in\Zd}$ is a stationary random field that is independent of $(\xi_v)_{v\in\Zd}$ and satisfies (\ref{eq:gammaY}). Then $(Y_vZ_v)$ satisfies Assumption~2. Furthermore, the extremal functional $\eta$ is given by
\begin{equation*}%\label{eq:MAfunctional}
	\eta((y_v)) 
	= \frac{\EE^{\bb{Y}}\Bigl[\max_{v\in \Zd}\bigl(p_\xi(\hat{Y}_v\psi_v)^\alpha_+
	+(1-p_\xi)(\hat{Y}_v\psi_v)^\alpha_-\bigr)_+\mid \hat{\Ii}\Bigr]((y_v))}
	{\sum_{v\in \Zd}
	\bigl(p_\xi(\psi_v)^\alpha_+ + (1-p_\xi)(\psi_v)^\alpha_-\bigr)} .
\end{equation*}
\end{lemma}

\begin{proof}
Recall the notation $B^{(m)}$ introduced in Section~\ref{sec:prelim}. Using this, we can write
\[
	Z_v^t = \sum_{u \in B^{(t)}} \psi_u \xi_{v-u} .
\]
We start by showing that $(Z_v^t)$ is regularly varying, which combined with Lemma~\ref{lem:MA1} shows that Assumption~2 is satisfied. Define $\bb{Z}^{t,m}=(Z_v^t)_{v\in B^{(m)}}$ and $\bb \xi^m = (\xi_v)_{v\in B^{(m)}}$ for all $m$. By definition, $\bb Z^{t,m}$ is a linear transformation of $\bb{\xi}^{t+m}$. Since $\bb{\xi}^{t+m}$ is i.i.d. it is jointly regularly varying with 
\[
\begin{aligned}
	&
	\frac{\PP(x^{-1}\bb{\xi}^{t+m}\in\cdot )}
	{\PP(\abs{\xi_0}>x)}\\
	&
	\stackrel{v}{\to}
	\alpha 
	\sum_{u\in B^{(t+m)}}
	\int_{\R\setminus\{0\}}
	\abs{x}^{-\alpha-1}\bigl(
	p_\xi\I{(0,\infty)}(x)+(1-p_\xi)\I{(-\infty,0)}(x)\bigr)
	\I{\{x\bb{1}_u\in\cdot\}}\dd x,
	\end{aligned}
\]
as $x\to\infty$,
where $\bb{1}_u$ is an element in $\R^{\abs{B^{(t+m)}}}$ with $0$'s in all indices except for index $u$, where it takes the value $1$. Now define $\bb{\psi}^u=(\psi_{-u+v})_{v\in B^{(m)}}$ with $\psi_u=0$ for $u\notin B^{(t)}$. Then, as $x\to\infty$,
\begin{equation}\label{eq:Zregvarconv}
\begin{aligned}
&
\frac{\PP(x^{-1}\bb{Z}^{t,m}\in\cdot )}{\PP(\abs{\xi_0}>x)}\\
&\stackrel{v}{\to}\alpha \sum_{u\in B^{(t+m)}}\int_{\R\setminus\{0\}}\abs{x}^{-\alpha-1}\big(p_\xi\I{(0,\infty)}(x)+(1-p_\xi)\I{(-\infty,0)}(x)\big)\I{\{x\bb{\psi}^u\in\cdot\}}\dd x.
\end{aligned}
\end{equation}
Recalling that $\norm{\bb{Z}^{t,m}}=\max_{v\in B^{(m)}}\abs{Z^{t}_v}$, we find from (\ref{eq:Zregvarconv}) that
\begin{equation}\label{eq:normZconv}
	\frac{\PP(\norm{\bb{Z}^{t,m}}>x)}{\PP(\abs{\xi_0}>x)}
	\stackrel{v}{\to}
	\sum_{u\in B^{(t+m)}}\norm{\bb{\psi}^u}^{\alpha}.
\end{equation}
Furthermore, combining (\ref{eq:Zregvarconv}) and (\ref{eq:normZconv}), we find that
\begin{equation}\label{eq:Zspectralconv}
\begin{aligned}
&\frac{\PP(\norm{\bb{Z}^{t,m}}>x,\bb{Z}^{t,m}/\norm{\bb{Z}^{t,m}}\in\cdot )}{\PP(\norm{\bb{Z}^{t,m}}>x)}\\
&\stackrel{v}{\to}\Big(\sum_{u\in B^{(t+m)}}\norm{\bb{\psi}^u}^{\alpha}\Big)^{-1} \sum_{u\in B^{(t+m)}}\norm{\bb{\psi}^u}^{\alpha}\big(p_\xi\I{\{\bb{\psi}^u/\norm{\bb{\psi}^u}\in\cdot\}}+(1-p_\xi)\I{\{-\bb{\psi}^u/\norm{\bb{\psi}^u}\in\cdot\}}\big).
\end{aligned}
\end{equation}
This shows that $(Z_v^t)$ is regularly varying, and thus Assumption~\ref{ass:randomfield} is satisfied. Let the family $\bb{\Theta}^{t,m}=(\Theta^{t,m}_v)_{v\in B^{(m)}}$ represent the limit distribution in (\ref{eq:Zspectralconv}), i.e. the spectral distribution. Then,
\[
	\EE(\Theta^{t,m}_0)_+^\alpha
	=\Big(\sum_{u\in B^{(t+m)}}\norm{\bb{\psi}^u}^{\alpha}\Big)^{-1}
	\sum_{u\in B^{(t+m)}}\big(p_\xi(\bb{\psi}^u_0)^\alpha_++(1-p_\xi)(\bb{\psi}^u_0)^\alpha_-\big),
\]
and 
\begin{align*}
	\MoveEqLeft
	\EE\Big[\Big((Y_0\Theta^{t,m}_0)_+^\alpha-(\max_{v\in A^{(m)}_0}Y_v\Theta^{t,m}_v)_+^\alpha\Big)_+\mid \bb{Y}=\hat{\bb{Y}}\Big]\\ &
	=\Big(\sum_{u\in B^{(t+m)}}\norm{\bb{\psi}^u}^{\alpha}\Big)^{-1}\sum_{u\in B^{(t+m)}}\Big[p_\xi(\hat{Y}_0\bb{\psi}^u_0)^\alpha_++(1-p_\xi)(\hat{Y}_0\bb{\psi}^u_0)^\alpha_- \\ &
	\quad 
	- \max_{v\in \Anm}\Big(p_\xi(\hat{Y}_v\bb{\psi}^u_v)^\alpha_++(1-p_\xi)(\hat{Y}_v\bb{\psi}^u_v)^\alpha_-\Big)\Big]_+.
\end{align*}
Thus, 
\begin{align*}
	\MoveEqLeft
	\frac{\EE^{\bb{Y}}\Big[\EE\Big[\Big((Y_0\Theta^{t,m}_0)_+^\alpha-(\max_{v\in A^{(m)}_0}Y_v\Theta^{t,m}_v)_+^\alpha\Big)_+\mid \bb{Y}=\hat{\bb{Y}}\Big]\mid\hat{\Ii}\Big]}
	{\EE(\Theta^{t,m}_0)_+^\alpha}\\ &
	=\Big(\sum_{u\in B^{(t+m)}}\big(p_\xi(\bb{\psi}^u_0)^\alpha_++(1-p_\xi)(\bb{\psi}^u_0)^\alpha_-\big)\Big)^{-1} \\ &
	\quad
	\times \sum_{u\in B^{(t+m)}}\EE^{\bb{Y}}\Big[\Big(p_\xi(\hat{Y}_0\bb{\psi}^u_0)^\alpha_++(1-p_\xi)(\hat{Y}_0\bb{\psi}^u_0)^\alpha_- \\ &
	\qquad
	-\max_{v\in \Anm}\Big(p_\xi(\hat{Y}_v\bb{\psi}^u_v)^\alpha_++(1-p_\xi)(\hat{Y}_v\bb{\psi}^u_v)^\alpha_-\Big)\Big)_+\mid \hat{\Ii}\Big]\\&
	=\Big(\sum_{u\in B^{(t+m)}}\big(p_\xi(\psi_{-u})^\alpha_++(1-p_\xi)(\psi_{-u})^\alpha_-\big)\Big)^{-1}\\&
	\quad
	\times \sum_{u\in B^{(t+m)}}\EE^{\bb{Y}}\Big[\Big(p_\xi(\hat{Y}_{-u}\psi_{-u})^\alpha_++(1-p_\xi)(\hat{Y}_{-u}\psi_{-u})^\alpha_-\\&
	\qquad
	-\max_{v\in \Anm}\Big(p_\xi(\hat{Y}_{-u+v}\psi_{-u+v})^\alpha_++(1-p_\xi)(\hat{Y}_{-u+v}\psi_{-u+v})^\alpha_-\Big)\Big)_+\mid \hat{\Ii}\Big],
\end{align*}
where we in the second equality have used that the conditional expectation given $\hat{\Ii}$ is invariant to translations. Recalling that we consider $\psi_u=0$ for all $u\in B^{(t)}$, this has the following limit as $m$ tends to infinity,
\[
	\frac{\EE^{\bb{Y}}\Big[\max_{v\in B^{(t)}}\Big(p_\xi(\hat{Y}_v\psi_v)^\alpha_++(1-p_\xi)(\hat{Y}_v\psi_v)^\alpha_-\Big)\Big)_+\mid \hat{\Ii}\Big]}
	{\sum_{v\in B^{(t)}}\big(p_\xi(\psi_v)^\alpha_++(1-p_\xi)(\psi_v)^\alpha_-\big)}.
\]
Finally, letting $t$ tend to infinity gives the desired result.
\end{proof}

\section{Example~\ref{ex:garch}: GARCH(1,1) volatility}\label{sec:garch}

Recall that we consider the one-dimensional case and let $(Z_v)_{v\in \ZZ}$ be given as the solution to
\[
	Z_v^2=\alpha_0+Z_{v-1}^2(\alpha_1\xi_{v-1}^2+\beta_1),
\]
where $(\xi_v)_{v\in\ZZ}$ is an i.i.d. sequence. We make further assumptions, as previously described, ensuring that $(Z_v)_{v\in\ZZ}$ is stationary, regularly varying, strongly mixing and satisfies the anti-clustering condition. The index of the regular variation is $2\alpha$, where $\alpha$ is the solution to $\EE(\alpha_1\xi_1^2+\beta_1)^\alpha=1$. For the result we combine notation and techniques from \cite{MikoschRezapour2013} with Corollary~\ref{cor:Zisstrongmixing}. For ease of notation, define for each $v\in\ZZ$
\[
	A_v=\alpha_1\xi_v^2+\beta_1 .
\]

\begin{lemma}
Let $(Z_v)_{v\in \ZZ}$ be defined as above and assume that $(Y_v)_{v\in\ZZ}$ is a stationary process, independent of $(\xi_v)_{v\in\ZZ}$, with $\EE\abs{Y_0}^{\gamma}<\infty$ for some $\gamma>2\alpha$. Then the extremal functional $\eta$ is given by $\eta((y_v))=\lim_{K\to\infty}\lim_{m\to\infty}\eta^{K,m}((y_v))$, where
\begin{align*}
&\eta^{K,m}((y_v))\\
&=\EE^{\bb{Y}}
	\Bigl[\EE\Bigl(\bigl((Y_0\I{\{\abs{Y_0}\le K\}})_+^\alpha
	-(\max_{1\leq v\leq m}Y_v\I{\{\abs{Y_0}\le K\}}\prod_{i=1}^v\sqrt{A_i})_+^\alpha \bigr)_+\mid\bb{Y}=\hat{\bb{Y}}\Bigr)\mid\hat{\Ii}\Bigr]((y_v)).
\end{align*}

\end{lemma}

\begin{proof}
We note that the assumptions of Corollary~\ref{cor:Zisstrongmixing} are satisfied. Define for each $m\in\NN$
\[
	\bb{Z}^m=(Z_{-m},\ldots,Z_m)
\]
and
\[
	\bb{R}^m=(R_{-m},\ldots,R_m)
	=\big(1, \sqrt{A_{-m}}, \sqrt{A_{-m}A_{-m+1}},\ldots, \sqrt{A_{-m}\cdots A_{m-1}}\big).
\]
From \cite[Lemma~4.2]{MikoschRezapour2013} we then have 
\[
	\frac{\PP(x^{-1}\bb{Z}^m\in{}\cdot{} )}{\PP(Z_0>x)}
	\stackrel{v}{\to} 
	2\alpha\int_0^\infty s^{-2\alpha-1}\PP(s\bb{R}^m\in\cdot)\,\dd s .
\]
Following arguments similar to \cite[Example~4.5]{MikoschRezapour2013} we  find that
\begin{align*}
	\frac{\PP(\norm{\bb{Z}^{m}}>x,\bb{Z}^{m}/\norm{\bb{Z}^{m}}\in{}\cdot{} )}{\PP(\norm{\bb{Z}^{m}}>x)} &
	\stackrel{v}{\to} 
	\frac{2\alpha\int_0^\infty s^{-2\alpha-1}\PP(s\norm{\bb{R}^m}\I{\{\bb{R}^m/\norm{\bb{R}^m}\in\cdot\}}>1)\,\dd s}{\EE\norm{\bb{R}^m}^{2\alpha}}\\ & 
	=\, 
	\frac{\EE\norm{\bb{R}^m}^{2\alpha}\I{\{\bb{R}^m/\norm{\bb{R}^m}\in\cdot\}}}{\EE\norm{\bb{R}^m}^{2\alpha}}.
\end{align*}
We let $\bb{\Theta}^m=(\Theta^m_v)_{\abs{v}\leq m}$ represent the limiting distribution. Then, using the independence between $(Y_v)$ and $(\xi_v)$ and the fact that each $A_v$ is non-negative with $\EE A_v^{\alpha}=1$, we find that the expression $\eta^{K,m}((y_v))$ as defined in Corollary~\ref{cor:Zisstrongmixing} equals the desired result.
\end{proof}

%%%%%%%%%%%%%%%%%%%%%%%%%%%%%%%%%%%%%%%%%%%%%%
%% Single Appendix:                         %%
%%%%%%%%%%%%%%%%%%%%%%%%%%%%%%%%%%%%%%%%%%%%%%
%\begin{appendix}
%\section*{???}%% if no title is needed, leave empty \section*{}.
%\end{appendix}
%%%%%%%%%%%%%%%%%%%%%%%%%%%%%%%%%%%%%%%%%%%%%%
%% Multiple Appendixes:                     %%
%%%%%%%%%%%%%%%%%%%%%%%%%%%%%%%%%%%%%%%%%%%%%%
\begin{appendix}

%%%%%%%%%%%%%%%%%%%%%%%%
\section{Ergodic theory}\label{app:ergodictheorem}
%%%%%%%%%%%%%%%%%%%%%%%%

Here we formulate a spatial version of the classical Birkhoff--Khinchin theorem. The theorem, as formulated in Lemma~\ref{lem:ergodictheorem} below, relies on results from \cite{Krengel1985}; see also Section~5 of \cite{StehrRonnNielsen2020} for a more detailed exposition of the application of ergodic theory to random fields, and, in particular, \cite[Theorem~9]{StehrRonnNielsen2020}, which has a similar but simpler formulation of the theorem.

Let $S$ be the set of all real-valued fields $(x_v)_{v\in\Zd}$, and let $\Bb(S)$ be the corresponding $\sigma$-algebra generated by all coordinate projections. Define the translation map in the $\ell$'th direction, $T_\ell:S\to S$, as
\[
T_\ell\big((x_{v_1,\ldots,v_d})_{(v_1,\ldots,v_d)\in\Zd}\big)=\big((x_{v_1,\ldots,v_{\ell-1},v_\ell+1,v_{\ell+1},\ldots,v_d})_{(v_1,\ldots,v_d)\in\Zd}\big)
\]
for each $\ell=1,\ldots,d$. Let $\hat{\Ii}$ be the invariant $\sigma$-algebra on $S$, i.e. the $\sigma$-algebra consisting of all sets in $\Bb(S)$ invariant to $T_i$ for $i=1,\ldots,d$.

For a random field $\bb{Y}=(Y_v)_v$ on $(S,\Bb(S))$ with distribution $\bb{Y}(\PP)$ and a measurable function $h: S\to\R$, we let $\EE^{\bb{Y}}(h\mid \hat{\Ii})$ denote the conditional expectation on $(S,\Bb(S),\bb{Y}(\PP))$
of $h$ given $\hat{\Ii}$. When the concrete definition of $h$ is relevant, the conditional expectation will be written on the form $\EE^{\bb{Y}}(h((\hat{Y}_v))\mid \hat{\Ii})$, where $\hat{\bb{Y}}=(\hat{Y}_v)_{v\in\Zd}$ is the identity map on $(S,\Bb(S))$.

It should be noted that $\EE^{\bb{Y}}(h\mid \hat{\Ii})$ is an $\hat{\Ii}$-measurable function $S\to\R$.

\begin{lemma}\label{lem:ergodictheorem}
Let $\bb{Y}=(Y_v)_{v\in\Zd}$ be a stationary random field, and assume that $(D_n)$ is a sequence satisfying Assumption~\ref{ass:geometry}. If $h: S\to\R$ is a measurable function satisfying $\EE \abs{h((Y_u)_{u\in \Zd})}<\infty$, then it holds that
\[
	\frac{1}{\abs{D_n}}\sum_{v\in D_n} h((y_{u+v})_{u\in \Zd})
	\to \EE^{\bb{Y}} (h\mid \hat{\Ii})((y_u)_{u\in \Zd})
\]
for almost all realizations $(y_v)_v$ of $(Y_v)_v$.

\end{lemma}
We say that $\bb{Y}=(Y_v)_v$ is ergodic, if $\hat{\Ii}$ is trivial. That is $\bb{Y}(\PP)(A)\in\{0,1\}$ for all $A\in \hat{\Ii}$. Note that if $\bb{Y}$ is ergodic, then the limit is non-random, i.e. independent of the specific realization of $\bb{Y}$.

\begin{proof}
In the case, where $(D_n)$ are boxes in $\NN_0^d$, Lemma~\ref{lem:ergodictheorem} is identical to \cite[Theorem~6.2.8]{Krengel1985}. Extending that theorem to boxes in $\Zd$ follows by dividing the index set along the coordinate axes into $2^d$ subsets and applying the result to each sequence of subsets separately. That the conditional expectation in the limit is the same in all sub-directions follows from the fact that, for each $\ell=1,\ldots,d$, invariance to $T_\ell$ is equivalent to invariance to $T_\ell^{-1}$. 

Extending the result to allow for more general sequences $(D_n)$, as given in Assumption~\ref{ass:geometry}, is possible, since the proof  of \cite[Theorem~6.2.8]{Krengel1985} only relies on \cite[Corollary~6.2.7]{Krengel1985}, which is clearly satisfied in the extended framework as well, and on the geometrical requirement
\begin{equation}\label{eq:krengelgeometricassumption}
	\frac{\abs{D_n\Delta(D_n+\bb e_\ell)}}{\abs{D_n}}\to 0
\end{equation}
for each $\ell=1,\ldots,d$. Here $A\Delta B=(A\cup B)\setminus (A\cap B)$ and $\bb e_\ell$ denotes the unit vector in direction $\ell$. 

To see that \eqref{eq:krengelgeometricassumption} is satisfied, let $K$ be the number of neighboring boxes in all directions, including via edges and corners, of the box $J_0^n$. Now define $E_n$ to be the union of $D_n^+\setminus D_n^-$ with all its neighboring boxes, and note that $\abs{E_n}\leq K\abs{D_n^+\setminus D_n^-}$. Also note that for $n$ large enough it holds that
\[
	D_n\Delta(D_n+\bb e_\ell)
	\subseteq E_n.
\]
Now \eqref{eq:krengelgeometricassumption} follows from \eqref{eq:geometricassumption}.
\end{proof}

%%%%%%%%%%%%%%%%%%%%%%%%
\section{Results on regular variation}\label{app:regvar}
%%%%%%%%%%%%%%%%%%%%%%%%
In this section we derive a limit of
\[
\sum_{v\in D_n}
	\PP \bigl( M_{yZ}(\Avm)\leq a_nx<y_vZ_v \bigr)
\]
for almost all realizations $(y_v)$ of $(Y_v)$, and a limit for the same expression with the approximating fields $(Z_v^t)$ for each $t\in\NN$. The proof strategy will be to use the ergodic result Lemma~\ref{lem:ergodictheorem} from Appendix~\ref{app:ergodictheorem} in combination with the spectral behavior of $(Z_v)$. There will be some similarities to the proof of the multivariate version of Breiman's lemma. See e.g. the proof of \cite[Proposition~A.1]{Basrak2002}. The results will be formulated and proven in terms of a general field $(R_v)$ acting as a place-holder for either $(Z_v)$ or $(Z_v^t)$.

\begin{lemma}\label{lem:fixedy}
Let $(R_v)_\vZ$ be a stationary and regularly varying random field with index $\alpha$, and, for each $m\in\NN$, let $(\Theta_v^m)_{v\in B^{(m)}}$ represent the spectral field of $\bb{R}^m = (R_v)_{v\in B^{(m)}}$. Furthermore, let $(y_v)_{v\in B^{(m)}}$ be real constants. If $(a_n)$ is chosen such that $\abs{D_n}^{-1}\sim \PP(R_0>a_n)$, then
\[
	\abs{D_n}\PP \bigl(M_{yR}(\Anm)>a_nx \bigr)
	\to x^{-\alpha} 
	\frac{\EE[(\max_{v\in \Anm}y_v\Theta_v^m)_+^\alpha]}
	{\EE(\Theta_0^m)_+^\alpha}
\]
as $n\to\infty$. 
\end{lemma}

\begin{proof}
The result is trivial in the case with $y_v=0$ for all $v\in B^{(m)}$, so it can without loss of generality be assumed that $y_v\neq 0$ for at least one $v\in B^{(m)}$. With standard extension arguments it follows from (\ref{eq:regvarmulti}) that
\begin{align*}
	\frac{\PP(M_{yR}(\Anm)>a_nx)}{\PP (\norm{\bb{R}^m}>a_n)} &
	=\frac{\PP\bigl(\max_{v\in \Anm}\big(y_v\frac{R_v}{\norm{\bb{R}^m}}\big)_+\norm{\bb{R}^m}>a_nx \bigr)}{\PP(\norm{\bb{R}^m}>a_n)} \\&
	\to x^{-\alpha}\EE((\max_{v\in \Anm}y_v\Theta_v^m)_+^\alpha)
\end{align*}
and
\begin{equation}\label{eq:tailnormR}
	\frac{\PP(R_0>a_n)}{\PP(\norm{\bb{R}^m}>a_n)}
	=
	\frac{\PP\bigl(\bigl(\tfrac{R_0}{\norm{\bb{R}^m}}\bigr)_+\norm{\bb{R}^m}>a_n)}{\PP(\norm{\bb{R}^m}>a_n)}
	\to \EE(\Theta_0^m)_+^\alpha .
\end{equation}
Now the desired result follows from the choice $\abs{D_n}\PP(R_0>a_n)\to 1$.
\end{proof}

\begin{lemma}\label{lem:extremalindexgivenyfixedm}
Let $(R_v)_\vZ$ be a stationary and regularly varying random field with index $\alpha$, and, for each $m\in\NN$, let $(\Theta_v^m)_{v\in B^{(m)}}$ represent the spectral field of $\bb{R}^m = (R_v)_{v\in B^{(m)}}$. Let $(Y_v)$ be a stationary random field such that $\EE\abs{Y_0}^\gamma<\infty$ for some $\gamma>\alpha$.  Assume that $(Y_v)$ is independent of $(R_v)$ and $(\Theta^{m}_v)_{v\in B^{(m)}}$ for all $m\in\NN$. If $(a_n)$ is chosen such that $\abs{D_n}^{-1}\sim \PP(R_0>a_n)$, then for almost all realizations $(y_v)$ of $(Y_v)$ it holds that
\[
	\sum_{v\in D_n}
	\PP \bigl( M_{yR}(\Avm)\leq a_nx<y_vR_v)
	\to x^{-\alpha}
	\eta_{\bb{R}^m}((y_v)),
\]
where
\begin{equation*}%\label{eq:eta}
	\eta_{\bb{R}^m}((y_v))
	=\frac{\EE^{\bb{Y}}
	\Big[\EE\Bigl(\bigl((Y_0\Theta^{m}_0)_+^\alpha-(\max_{v\in \Anm} Y_v\Theta^{m}_v)_+^\alpha \bigr)_+\mid\bb{Y}=\hat{\bb{Y}}\Bigr)\mid\hat{\Ii}\Big]((y_v))}
	{\EE(\Theta_0^{m})_+^\alpha} .
\end{equation*}
\end{lemma}

\begin{proof}
We write
\begin{equation}\label{eq:divideintwo}
\begin{aligned}
	\MoveEqLeft
	\sum_{v\in D_n}\PP(M_{yR}(A^{(m)}_v)\leq a_nx<y_vR_v)\\&
	=\sum_{v\in D_n}\PP(M_{yR}(A^{(m)}_v\cup\{v\})> a_nx)-
	\sum_{v\in D_n}\PP(M_{yR}(A^{(m)}_v)> a_nx)
\end{aligned}
\end{equation}
and handle the two terms separately but in the same way. To keep notation simple, we focus on the second term. 
Let $\bb y_v^m = (y_z)_{z \in B_v^{(m)}}$ and $\bb R_v^m = (R_z)_{z \in B_v^{(m)}}$ for each $m\in\NN$ and $v\in\Zd$, where $B_v^{(m)}$ is defined in Section~\ref{sec:prelim}. For $K\in\NN$ we now have
\begin{equation}\label{eq:twoparts}
\begin{aligned}
	\MoveEqLeft
	\sum_{v\in D_n}\PP(M_{yR}(A^{(m)}_v)> a_nx)\\&
	=\sum_{v\in D_n}\I{\{\norm{\bb{y}^{m}_v}< K\}} \PP(M_{yR}(A^{(m)}_v)> a_nx) \\ & \quad
	+\sum_{v\in D_n}\I{\{\norm{\bb{y}^{m}_v}\geq K\}}\PP(M_{yR}(A^{(m)}_v)> a_nx).
\end{aligned}
\end{equation}
For the second term we find for $n$ large that
\begin{equation}\label{eq:aboveMpart}
\begin{aligned}
	\MoveEqLeft
	\sum_{v\in D_n}\I{\{\norm{\bb{y}^{m}_v}\geq K\}}
	\PP(M_{yR}(A^{(m)}_v)> a_nx)\\&
	\le \sum_{v\in D_n}\I{\{\norm{\bb{y}^{m}_v}\geq K\}}
	\PP(\norm{\bb{y}^{m}_v}\norm{\bb{R}^{m}_v} > a_nx)\\&
	\le \sum_{v\in D_n}\I{\{\norm{\bb{y}^{m}_v}\geq K\}}
	\PP(\norm{\bb{R}^{m}_0} > a_nx)c \norm{\bb{y}^{m}_v}^\gamma ,
\end{aligned}
\end{equation}
where we in the second inequality have used the stationarity of $(R_v)$ and the fact that the tail of $\norm{\bb{R}^{m}_0}$ is regularly varying: There exist $c>0$ and $x_0>0$ such that
\[
	\frac{\PP(\norm{\bb{R}^{m}_0}>x/y)}
	{\PP(\norm{\bb{R}^{m}_0}>x)}
	\leq c y^\gamma
\]
for all $x\geq x_0$ and $y\geq 1$. Utilizing \eqref{eq:tailnormR} and the choice of $(a_n)$, we find from Lemma~\ref{lem:ergodictheorem} that the upper bound in \eqref{eq:aboveMpart} has limit
\[	
	x^{-\alpha}C\frac{\EE^{\bb{Y}}
	\Bigl[\I{\{\norm{\hat{\bb{Y}}^{m}_0}\geq K\}}
	\norm{\hat{\bb{Y}}^{m}_0}^\gamma\mid \hat{\Ii}\Bigr]((y_v))}
	{\EE(\Theta^{m}_0)_+^\alpha} 
\]
as $n\to\infty$.
By dominated convergence, this tends to 0 as $K\to\infty$, since $\EE\norm{\bb{Y}^{m}_0}^\gamma<\infty$.

We now consider the first term in \eqref{eq:twoparts}. When $\norm{\bb y_v^m}<K$, the field $\bb y_v^m$ takes its values in the space $\mathcal{S}=[-K,K)^{\abs{B^{(m)}}}$. Now fix $p\in\NN$ and construct the set $\mathcal{A}_p$ of grid points of $\mathcal S$ with a distance $K/p$ apart, i.e.
\[
	\mathcal{A}_p = (\tfrac{K}{p} \ZZ)^{^{\abs{B^{(m)}}}} \cap \mathcal S .
\]
We only highlight the dependence on $p$, as $K$ will be fixed throughout the use of the set.
For each point $\bb x_p \in \mathcal A_p$, we let $\mathcal B_{\bb x_p}$ denote the box in $\R^{\abs{B^{(m)}}}$ with ``left'' corner point $\bb x_p$,
\[
	\mathcal B_{\bb x_p} 
	= \bb x_p +  \bigl[0,K/p\bigr)^{\abs{B^{(m)}}} .
\]
These boxes make up a partition of the set $\mathcal S$, and thus when $\norm{\bb y_v^m}<K$, the field $\bb y_v^m$ lies in exactly one such box. As a last piece of convenient notation, we let $\xu$ and $\xl$ denote the coordinate-wise numerically largest respectively smallest value in (the closure of) $\mathcal B_{\bb x_p}$. That is, with $x_{p,z}$ denoting the $z$'th coordinate of $\bb x_p$ and similarly for $\xu$ and $\xl$,
\begin{equation*}
	\overline{x}_{p,z} 
	= x_{p,z} + \tfrac{K}{p} \I{\{x_{p,z}\ge 0\}}
	\qquad\text{and}\qquad
	\underline{x}_{p,z} 
	= x_{p,z} + \tfrac{K}{p} \I{\{x_{p,z}< 0\}} .
\end{equation*} 

Having introduced the notation above, we can write the first term in \eqref{eq:twoparts} as
\begin{equation}\label{eq:sumivided}
	\sum_{\bb{x}_p\in\mathcal{A}_p}
	\sum_{v\in D_n}
	\I{\{\bb{y}^{m}_v\in \mathcal B_{\bb x_p}\}}
	\PP\bigl(M_{yR}(\Avm)> a_nx\bigr).
\end{equation}
This is bounded from above by 
\[
	\sum_{\bb{x}_p\in\mathcal{A}_p}
	\sum_{v\in D_n}
	\I{\{\bb{y}^{m}_v\in \mathcal B_{\bb x_p}\}}
	\PP\bigl(M_{\xu R}(\Avm)> a_nx\bigr),
\]
which by Lemmas~\ref{lem:ergodictheorem} and \ref{lem:fixedy} converges to
\begin{equation}\label{eq:upperboundlimit}
	x^{-\alpha}
	\sum_{\bb{x}_p\in\mathcal{A}_p}
	\frac{
	\EE^{\bb{Y}}\bigl(
	\I{\{\hat{\bb{Y}}_0^{m}\in \mathcal B_{\bb x_p}\}}\mid\hat{\Ii}\bigr)((y_v))
	\,\EE(\max_{v\in A_0^{(m)}}
	\overline x_{p,v} \Theta^{m}_v)_+^\alpha}{\EE(\Theta^{m}_0)_+^\alpha}
\end{equation}
as $n\to\infty$. Similarly, (\ref{eq:sumivided}) is bounded from below by an expression that tends to
\begin{equation}\label{eq:lowerboundlimit}
	x^{-\alpha}
	\sum_{\bb{x}_p\in\mathcal{A}_p}
	\frac{
	\EE^{\bb{Y}}\bigl(
	\I{\{\hat{\bb{Y}}_0^{m}\in \mathcal B_{\bb x_p}\}}\mid\hat{\Ii}\bigr)((y_v))
	\,\EE(\max_{v\in A_0^{(m)}}
	\underline x_{p,v} \Theta^{m}_v)_+^\alpha}{\EE(\Theta^{m}_0)_+^\alpha}
\end{equation}
as $n\to\infty$. Clearly, both (\ref{eq:upperboundlimit}) and (\ref{eq:lowerboundlimit}) tends to
\[
x^{-\alpha}\frac{\EE^{\bb{Y}}\Big[\EE\Big(\I{\norm{\bb{Y}^{m}_0}<M}(\max_{v\in A_0^{(m)}}Y_v\Theta^{m}_v)_+^\alpha\mid \bb{Y}=\hat{\bb{Y}}\Big)\mid \hat{\Ii}\Big]((y_v))}{\EE(\Theta^{m}_0)_+^\alpha}
\]
as $p\to\infty$. Letting $K\to\infty$, using that $\EE\norm{\bb{Y}^{m}_0}^\alpha<\infty$, gives that
\begin{align*}
	\MoveEqLeft	
	\sum_{v\in D_n}\PP(M_{yR}(A^{(m)}_v)> a_nx) \\ &
	\to x^{-\alpha}\frac{\EE^{\bb{Y}}
	\Bigl[\EE\Bigl((\max_{v\in A^{(m)}_0}Y_v\Theta^{m}_v)_+^\alpha\mid 
	\bb{Y}=\hat{\bb{Y}}\Bigr)\mid \hat{\Ii}\Bigr]
((y_v))}
{\EE(\Theta^{m}_0)_+^\alpha}.
\end{align*}
Handling the first term in (\ref{eq:divideintwo}) similarly, gives the desired result.
\end{proof}

%%%%%%%%%%%%%%%%%%%%%%%%
\section{Extremal representation}
\label{app:Geometry}
%%%%%%%%%%%%%%%%%%%%%%%%

Throughout this appendix we assume that $(D_n)$ satisfies Assumption~\ref{ass:geometry} and that $(Z_v)$, with approximating fields $(Z_v^t)$, and $(Y_v)$ satisfy Assumption~\ref{ass:randomfield}. We set out to prove (in Lemma~\ref{lem:JsetandAset} below) that
\begin{equation*}
	\PP\bigl(\max_{v \in D_n} y_v Z_v \le a_n x \bigr)  
	=
	\exp\Bigl(-\sum_{v\in D_n} \PP \bigl( M_{yZ}(\Av) \le a_n x < y_v Z_v \bigr) \Bigr) + o(1)
\end{equation*}
for almost all realizations $(y_v)$ of $(Y_v)$, which is also satisfied for the approximating fields $(Z_v^t)$ for each $t\in\NN$.
The proof follows by arguments similar to those of \cite{RNielsenStehr2022}, however with some modifications due to lack of stationarity caused by conditioning on $(Y_v)=(y_v)$. This is handled using the ergodic-like convergence of Lemma~\ref{lem:ergodictheorem} in Appendix~\ref{app:ergodictheorem}.

The results herein will, again, be formulated and proven in terms of a general field $(R_v)$ acting as a place-holder for either $(Z_v)$ or $(Z_v^t)$. Therefore, by assumption and by Lemma~\ref{lem:ratiobound}, we have that there exist $n_0\in\NN$ and $C_0>0$ such that
\begin{equation}\label{eq:ratioboundapp}
	\PP(y R_0> a_n x) \le \frac{C_0\, x^{-\alpha}}{\abs{D_n}} \abs{y}^\gamma
\end{equation}
for all $n\ge n_0$ and all $\abs{y}>1$.

To ease notation throughout, we let $P_n$ and $Q_n$ be the set of indices used in the geometrical approximation $\dnm \subseteq D_n \subseteq \dnp$ of $D_n$ appearing in Assumption~\ref{ass:geometry}, i.e.
\[
	P_n = \bigl\{
	z\in\Zd \::\: \Jz \subseteq D_n
	\bigr\}
\]
and
\[
	Q_n = \bigl\{
	z\in\Zd \::\: \Jz \cap D_n \neq \emptyset
	\bigr\}
\]
for all $n\in\NN$. Hence, $\bigcup_{P_n} \Jz = \dnm$ and $\bigcup_{Q_n} \Jz = \dnp$, and note that the assumption \eqref{eq:geometricassumption} is equivalent to the asymptotic equivalences
\[
	\abs{P_n} \sim \abs{Q_n} \sim \frac{\abs{D_n}}{\abs{J_z^n}}
\]
as $n\to\infty$.

The following lemma, which states that the extremal behavior of $(y_v R_v)$ on a collection of $\Jz$-sets is essentially equal to the extremal behavior on corresponding subsets of asymptotically equivalent sizes, will be useful throughout the section.

\begin{lemma}\label{lem:tailapprox}
Let $(R_v)$ be either $(Z_v)$ or $(Z_v^t)$ for some fixed $t\in\NN$. For $z\in\Zd$ let $H_z^n \subseteq \Jz$ be a set which expands in size as $\Jz$, i.e. $\abs{H_z^n} \sim \abs{\Jz}$ as $n\to\infty$. Then, for almost all realizations $(y_v)$ of $(Y_v)$,
\begin{equation}\label{eq:subsetofJz}
	\sum_{z \in \Pn} \PP \bigl(
	M_{yR}(\Jz \setminus H_z^n) > a_n x \bigr)	
	\to 0
\end{equation}
as $n\to\infty$. The result is also true if $P_n$ is replaced by $Q_n$. Moreover, 
\begin{equation}\label{eq:sumapprox1}
	\sum_{v \in \dnp} \PP \bigl(
	y_vR_v > a_n x \bigr)-	
	\sum_{v \in \dnm} \PP \bigl(
	y_vR_v > a_n x \bigr)	
	\to 0
\end{equation}
and
\begin{equation}\label{eq:sumapprox2}
	\sum_{z \in \Qn} \PP \bigl(
	M_{yR}(\Jz) > a_n x \bigr)	-
	\sum_{z \in \Pn} \PP \bigl(
	M_{yR}(\Jz) > a_n x \bigr)	
	\to 0
\end{equation}
as $n\to\infty$.
\end{lemma}

\begin{proof}
It is enough to show \eqref{eq:subsetofJz} with summation over $P_n$, as the case of $Q_n$ follows identically.

Let $b\ge 1$ be given. Then
\begin{align*}
	\MoveEqLeft \sum_{z \in \Pn} \PP \bigl(
	M_{y R}(\Jz \setminus H_z^n) > a_n x \bigr) \\ &
	\le \sum_{z \in \Pn} \sum_{v\in \Jz\setminus H_z^n} 
	\PP \bigl(
	y_v R_v > a_n x \bigr) \\ &
	= \sum_{z \in \Pn} \sum_{v\in \Jz\setminus H_z^n} 
	\PP \bigl(
	y_v R_v > a_n x \bigr) \I{\{\abs{y_v}\le b\}} +
	\sum_{z \in \Pn} \sum_{v\in \Jz\setminus H_z^n} 
	\PP \bigl(
	y_v R_v > a_n x \bigr) \I{\{\abs{y_v}> b\}} .
\end{align*} 
We consider the two terms separately. For the former term we see by stationarity of $(R_v)$ that
\begin{align*}
	\MoveEqLeft 
	\sum_{z \in \Pn} \sum_{v\in \Jz\setminus H_z^n} 
	\PP \bigl(
	y_v R_v > a_n x \bigr) \I{\{\abs{y_v}\le b\}} \\ &
	\le \sum_{z \in \Pn} \sum_{v\in \Jz\setminus H_z^n} 
	\PP \bigl(
	\abs{R_v} > a_n x/b \bigr) \I{\{\abs{y_v}\le b\}} \\ &
	\le \abs{P_n}\, \abs{\Jz\setminus H_z^n} \,
	\PP \bigl(
	\abs{R_0} > a_n x/b \bigr) \\ &
	\to 0
\end{align*}
as $n\to\infty$ due to the choice of $a_n$ and the tail-balance of $(R_v)$. For the second term above we use \eqref{eq:ratioboundapp} and Lemma~\ref{lem:ergodictheorem} of Appendix~\ref{app:ergodictheorem} to obtain (for sufficiently large $n$)
\begin{align*}
	\MoveEqLeft	
	\sum_{z \in \Pn} \sum_{v\in \Jz\setminus H_z^n} 
	\PP \bigl(
	y_v R_v > a_n x \bigr) \I{\{\abs{y_v}> b\}} \\ &
	\le 
	\sum_{z \in \Pn} \sum_{v\in \Jz} 
	\PP \bigl(
	y_v R_v > a_n x \bigr) \I{\{\abs{y}> b\}} \\ &
	= \sum_{v\in \dnm} 
	\PP \bigl(
	y_v R_v > a_n x \bigr) \I{\{\abs{y}> b\}} \\ &
	\le \frac{C_0\, x^{-\alpha}}{\abs{\dnm}}
	\sum_{v\in \dnm} \abs{y_v}^\gamma \I{\{\abs{y_v}> b\}} \\ &
	\to C_0\, x^{-\alpha} \EE^{\bb{Y}}[\abs{\hat{Y}_0}^\gamma \I{\{\abs{\hat{Y}_0}>b\}}\mid \hat{\Ii}]((y_v))
\end{align*}
as $n\to\infty$. Letting $b\to\infty$ proves \eqref{eq:subsetofJz}.

The claims \eqref{eq:sumapprox1} and \eqref{eq:sumapprox2} follow by similar arguments utilizing that $P_n \subseteq Q_n$, $\dnm \subseteq \dnp$ and $\abs{\dnp \setminus \dnm} = o(\abs{D_n})$.
\end{proof}
In the proof of Lemma~\ref{lem:maxvsJkasse} below, we will use an asymptotic equivalence between a product of probabilities, where each factor is $y_v$-dependent, and the exponential of the sum of the opposite probabilities. In the process of obtaining this, the next result will be useful. The asymptotic equivalence can be seen as the the $y_v$-dependent counterpart of the formula $a_n^n = \exp(-n(1-a_n)) + o(1)$ used in classical extremal theory, cf. \cite[Formula~(2.8)]{Obrien1987}.

\begin{lemma}\label{lem:expappr0}
Let $(R_v)$ be either $(Z_v)$ or $(Z_v^t)$ for some fixed $t\in\NN$. For almost all realizations $(y_v)$ of $(Y_v)$,
\[
	\sum_{z\in\Pn} \sum_{j=2}^\infty \frac{1}{j} 
	\Bigl(
	\sum_{v\in\Jz} \PP(y_v R_v >a_n x)
	\Bigr)^j
	\to 0
\]
as $n\to\infty$. The result is also true if $P_n$ is replaced by $Q_n$.
\end{lemma}

\begin{proof}
Let $b\ge 1$ large enough be given. Later in the proof we will explicitly choose $b$ such that
\[
	2 C_0\, x^{-\alpha} \EE^{\bb{Y}}[\abs{\hat{Y}_0}^\gamma \I{\{\abs{\hat{Y}_0}>b\}}\mid \hat{\Ii}]((y_v)) < 1,
\]
where $C_0$ is the constant from \eqref{eq:ratioboundapp}. This $b$ can be chosen for almost all $(y_v)$ due to the integrability of $\abs{Y_0}^\gamma$. 

By considering the cases $\abs{y_v}\le b$ and $\abs{y_v}>b$ separately as in the previous proof, we can write
\begin{align*}
	\sum_{v\in\Jz} \PP\bigl(y_v R_v >a_n x\bigr)
	\le \abs{\Jz}\, \PP\bigl(\abs{R_0}> a_n x/b \bigr)
	+ \sum_{v\in\Jz} 
	\PP\bigl( y_v R_0>a_n x \bigr) \I{\{\abs{y_v} > b\}} .
\end{align*}
Since $(a_1+a_2)^j \le 2^j (a_1^j + a_2^j)$ for positive $a_1,a_2$, The result especially follows if the following can be proven true for the considered realization $(y_v)$:
\begin{align}
	\label{eq:udvidetergode1}
	&
	\limsup_{n\to\infty}\sum_{j=2}^\infty \frac{1}{j} \abs{P_n} 
	\Bigl( 2\, \abs{\Jz}\, \PP(\abs{R_0}> a_n x/b)\Bigr)^j = 0 , \\ 
	\label{eq:udvidetergode2}
	&
	\lim_{b\to\infty}\limsup_{n\to\infty}
	\sum_{j=2}^\infty \frac{1}{j} \sum_{z \in \Pn} 
	\Bigl( 2\sum_{v\in\Jz} \PP\bigl( y_v R_0>a_n x \bigr) 
	\I{\{\abs{y_v} > b\}}
	\Bigr)^j = 0 .
\end{align}

We start by showing \eqref{eq:udvidetergode1}. Due to the choice of $a_n$ and the tail-balance of $R_0$ we have $\limsup_n \abs{D_n} \PP(\abs{R_0} > a_n x/b)<\infty$ for all $x>0$. In particular we find for any $j\ge 2$ that
\[
	\abs{P_n} 
	\Bigl( 2\, \abs{\Jz}\, \PP(\abs{R_0}> a_n x/b)\Bigr)^j
	\sim 
	\Bigl( 2\, \abs{D_n}\, \PP(\abs{R_0}> a_n x/b)\Bigr)^j 
	\frac{\abs{P_n}}{\abs{P_n}^j} 
	\to 0
\]
as $n\to\infty$. Moreover, for $n$ sufficiently large,
\begin{align*}
	\abs{P_n}^{1/j}\, 2\, \abs{\Jz}\, \PP(\abs{R_0}> a_n x/b)  &
	\le \abs{P_n}^{1/2}\, 2\, \abs{\Jz}\, \PP(\abs{R_0}> a_n x/b) \\ &
	\sim \frac{2\, \abs{D_n}\,\PP(\abs{R_0}> a_n x/b)}{\sqrt{\abs{P_n}}}
	< 1
\end{align*}
for all $j\ge 2$. Since $\sum_{j=2}^\infty \tfrac1j a^j<\infty$ exactly when $\abs{a}<1$, we conclude \eqref{eq:udvidetergode1} by dominated convergence.

To show \eqref{eq:udvidetergode2} we use the inequality $\sum_i a_i^j \le (\sum_i a_i)^j$ to obtain
\begin{align*}
	\sum_{z\in\Pn} 
	\Bigl( 2\sum_{v\in\Jz} \PP\bigl( y_v R_0>a_n x \bigr) 
	\I{\{\abs{y_v} > b\}}
	\Bigr)^j &
	\le 
	\Bigl( 2\sum_{z\in\Pn} \sum_{v\in\Jz} \PP\bigl( y_v R_0>a_n x \bigr) 
	\I{\{\abs{y_v} > b\}}\Bigr)^j \\ &
	= \Bigl( 2 \sum_{v \in\dnm} \PP\bigl( y_v R_0>a_n x \bigr) 
	\I{\{\abs{y_v} > b\}}
	\Bigr)^j .
\end{align*}
Exactly as in the previous proof we find that
\[
	\limsup_{n\to\infty} 2 \sum_{v \in\dnm} \PP\bigl( y_v R_0>a_n x \bigr) \I{\{\abs{y_v} > b\}}
	\le 
	2\, C_0\, x^{-\alpha} \EE^{\bb{Y}}[\abs{\hat{Y}_0}^\gamma \I{\{\abs{\hat{Y}_0}>b\}}\mid \hat{\Ii}]((y_v)) .
\]
Since $b$ is chosen such that this limit is $<1$, we find by dominated convergence that
\begin{align*}
	\MoveEqLeft
	\limsup_{n\to\infty}
	\sum_{j=2}^\infty \frac1j
	\sum_{z\in\Pn} 
	\Bigl( 2\sum_{v\in\Jz} \PP\bigl( y_v R_0>a_n x \bigr) 
	\I{\{\abs{y_v} > b\}}
	\Bigr)^j \\ &
	\le 
	\limsup_{n\to\infty}
	\sum_{j=2}^\infty \frac{1}{j}  
	\Bigl( 2\sum_{z\in\Pn} \sum_{v\in\Jz} \PP\bigl( y_v R_0>a_n x \bigr) 
	\I{\{\abs{y_v} > b\}}\Bigr)^j \\ &
	\le \sum_{j=2}^\infty \frac{1}{j}  
	\Bigl( 2\,C_0\, x^{-\alpha} \EE^{\bb{Y}}[\abs{\hat{Y}_0}^\gamma \I{\{\abs{\hat{Y}_0}>b\}}\mid \hat{\Ii}]((y_v))
	\Bigr)^j.
\end{align*}
Letting $b\to\infty$ the convergence \eqref{eq:udvidetergode2} follows.
\end{proof}
The next lemma, which will be used in the two subsequent results, provides the necessary approximate mixing of $(R_v)$. In its formulation, $\overline B_z^n$ denotes the union of the $3^d-1$ neighboring $J^n$-boxes of $\Jz$, i.e.
\begin{equation}\label{eq:overlineBz}
	\overline B_z^n = 
	\{u \in \Zd \setminus \Jz\::\:
	(z_\ell - 1) t_{n,\ell} 
	\le u_\ell
	\le (z_\ell + 2) t_{n,\ell} - 1,\ \text{for all } \ell
	\} .
\end{equation}

\begin{lemma}\label{lem:approxmixing}
Let $(R_v)$ be either $(Z_v)$ or $(Z_v^t)$ for some fixed $t\in\NN$. There exists $\bb \gamma_n \to \infty$ with $\bb \gamma_n = o(\bb t_n)$ such that
\begin{equation}\label{eq:approxmixing1}
	\abs[\Big]{
	\PP \Bigl(
		\bigcap_{z\in P_n} \bigl\{
		M_{yR}(\Hz) \le a_n x
		\bigr\}
	\Bigr) - 
	\prod_{z\in P_n} 
	\PP \bigl( M_{yR}(\Hz) \le a_n x \bigr)}
	\to 0
\end{equation}
as $n\to\infty$
for all $\bb \gamma_n$-separated sets $\Hz \subseteq \Jz$.
Moreover, for all $\bb \gamma_n$-separated sets $H_z^{n}\subseteq \Jz$ and $B_z^n \subseteq \overline B_z^n$,
\begin{equation}\label{eq:approxmixing2}
\begin{aligned}	
	\abs[\Big]{ &
	\sum_{z\in P_n} \PP \bigl(
		M_{yR}(\Hz) > a_n x, M_{yR}(B_z^n) > a_n x 
	\bigr) \\ & \quad
	- 
	\sum_{z\in P_n} 
	\PP \bigl( M_{yR}(\Hz) > a_n x \bigr)
	\PP \bigl( M_{yR}(B_z^n) > a_n x \bigr)}
	\to 0
\end{aligned}
\end{equation}
as $n\to\infty$. Both results remain true if $P_n$ is replaced by $Q_n$.
\end{lemma}

\begin{proof}
If $(R_v)=(Z_v^t)$ for some $t\in\NN$ both results follow easily from Assumption~\ref{ass:randomfield}\ref{eq:fieldass2}.

The results for $(R_v)=(Z_v)$ follow by appropriately approximating $(Z_v)$ with $(Z_v^t)$ and utilizing its strong mixing.
Let $n$ be given and let for each $t\in\NN$ the function $\alpha^t(\cdot)$ denote the strong mixing rate of $(Z_v^t)$. According to Assumption~\ref{ass:randomfield}\ref{eq:fieldass2}, let $\bb \gamma^t_n\to \infty$ satisfy $\bb \gamma_n^t=o(\bb t_n)$ and define 
\[
	\alpha^t_n : = \alpha^t(\bb \gamma_n^t)
\]
which satisfies $\abs{P_n} \alpha_n^t \to 0$. Moreover, let
\[
	\delta(t,n)
	= \sum_{v\in D_n^+} \PP(y_v Z_v^t>a_n x,\, y_v Z_v \le a_n x )
	+ \sum_{v\in D_n^+} \PP(y_v Z_v>a_n x,\, y_v Z_v^t \le a_n x) ,
\]
which, due to \eqref{eq:ratiobound2}, satisfies
\[
	\limsup_{t\to\infty} \limsup_{n\to\infty}\delta(t,n) = 0 .
\]

Let $(N_t^1)_{t\in\NN},\ (N_t^2)_{t\in\NN}$ and $(N_t^3)_{t\in\NN}$ be increasing sequences chosen such that
\begin{equation}\label{eq:approxmixinproof1}
\begin{aligned}
	\bb \gamma_n^{t}/ \bb t_n & \le 1/{t} \qquad \forall n\ge N_t^1, \\
	\abs{P_n} \alpha_n^t & \le 1/{t} \qquad \forall n\ge N_t^2,\\
	\delta(t,n) & \le 1/{t} \qquad \forall n\ge N_t^3.
\end{aligned}
\end{equation}
Using these, define
\[
	r_n^i = \sum_{t=1}^\infty t \I{[N_{t}^i,N_{t+1}^i)}(n)
	\qquad \text{for }i=1,2,3 ,
\]
and $r_n = \min\{r_n^1,r_n^2,r_n^3\}\to\infty$. Then
\[
	N_{t_n}^i 
	\le N_{r_n}^i
	\le N_{r_n^i}^i
	\le n
\]
for all sequences $(t_n)_{n\in\NN}$ where $t_n\le r_n$ (not to be confused with the vector $\bb t_n$). In particular, all the inequalities in \eqref{eq:approxmixinproof1} are satisfied for all $n$ if $t$ is replaced by $t_n$ such that $t_n\le r_n$. Hence, for all such $(t_n)$,
\begin{equation}\label{eq:approxmixinproof2}
	\bb \gamma_n^{t_n} / \bb t_n \to 0,
	\qquad
	\abs{P_n} \alpha_n^{t_n} \to 0,
	\qquad
	\delta(t_n,n) \to 0
\end{equation}
as $n\to\infty$. Now let such a sequence be fixed and define
\[
	\bb \gamma_n = \bb \gamma_n^{t_n}
	\quad\text{and}\quad
	\alpha_n = \alpha_n^{t_n} = \alpha^{t_n}(\bb \gamma_n^{t_n}) .
\]
This $\bb \gamma_n$ fulfills the requirement of the lemma.

Fix $n\in\NN$ and let the sets $\Hz$ appearing in \eqref{eq:approxmixing1} be $\bb \gamma_n$-separated. Using the strong mixing of $(Z_v^{t_n})$ recursively (in the last inequality), we find that
\begin{align*}
	\MoveEqLeft	
	\abs[\Big]{
	\PP \Bigl(
		\bigcap_{z\in P_n} \bigl\{
		M_{yZ}(\Hz) \le a_n x
		\bigr\}
	\Bigr) - 
	\prod_{z\in P_n} 
	\PP \bigl( M_{yZ}(\Hz) \le a_n x \bigr)} \\ 
	\le \,& 
	\abs[\Big]{
	\PP \Bigl(
		\bigcap_{z\in P_n} \bigl\{
		M_{yZ}(\Hz) \le a_n x
		\bigr\}
	\Bigr) - 
	\PP \Bigl(
		\bigcap_{z\in P_n} \bigl\{
		M_{yZ^{t_n}}(\Hz) \le a_n x
		\bigr\}
	\Bigr)} \\ &
	+ 
	\abs[\Big]{
	\prod_{z\in P_n} 
	\PP \bigl( M_{yZ}(\Hz) \le a_n x \bigr) - 
	\prod_{z\in P_n} 
	\PP \bigl( M_{yZ^{t_n}}(\Hz) \le a_n x \bigr)}\\ &
	+ 
	\abs[\Big]{
	\PP \Bigl(
		\bigcap_{z\in P_n} \bigl\{
		M_{yZ^{t_n}}(\Hz) \le a_n x
		\bigr\}
	\Bigr) - 
	\prod_{z\in P_n} 
	\PP \bigl( M_{yZ^{t_n}}(\Hz) \le a_n x \bigr)}\\ 
	\le &\,
	2\, \delta(n,t_n)+ \bigl(\abs{P_n}-1\bigr)\alpha_n,
\end{align*}
which converges to $0$ as $n\to\infty$ by \eqref{eq:approxmixinproof2}. This shows \eqref{eq:approxmixing1}.

The convergence \eqref{eq:approxmixing2} follows by similar arguments as above approximating $(Z_v)$ with $(Z_v^t)$ and utilizing that $\abs{P_n} \alpha_n \to 0$ as $n\to\infty$.
\end{proof}

\begin{lemma}\label{lem:maxvsJkasse}
Let $(R_v)$ be either $(Z_v)$ or $(Z_v^t)$ for some fixed $t\in\NN$. For almost all realizations $(y_v)$ of $(Y_v)$,
\begin{equation}\label{eq:maxprobequality}
	\PP\bigl(\max_{v \in D_n} y_v R_v \le a_n x \bigr)  
	=
	\exp\Bigl(-\sum_{z\in \Pn}\PP \bigl( M_{yR}(\Jz) > a_n x \bigr) \Bigr) + o(1)
\end{equation}
as $n\to\infty$. In particular,
\begin{equation}\label{eq:maxprobinequality}
	\PP\bigl(\max_{v \in D_n} y_v R_v \le a_n x \bigr)  
	\ge
	\exp\Bigl(-\sum_{v\in D_n}\PP \bigl( y_v R_v > a_n x \bigr) \Bigr) + o(1) .
\end{equation}
\end{lemma}

\begin{proof}
For any $z\in\Zd$ we define $\Hz \subseteq\Jz$ to be the subset of $\Jz$ with distance $\bb \gamma_n$ to the ``right'' boundary,
\[
	\Hz
	= \bigl\{
		u\in\Zd\::\:
		z_\ell t_{n,\ell} \le u_\ell \le 
		(z_\ell+1) t_{n,\ell} - 1 - \gamma_{n,\ell},\ 
		\text{for all }\ell=1,\dots,d
	\bigr\}, 
\]
where $\bb \gamma_n$ is the vector given in Lemma~\ref{lem:approxmixing}.
It will be used shortly that $\Hz$ is asymptotically of the same size as $\Jz$, i.e. $\abs{\Hz}\sim\abs{\Jz}$.
First we see that
\begin{align*}
	0 & 
	\le 
	\PP \Bigl(
		\bigcap_{z\in P_n} \bigl\{
		M_{yR}(\Hz) \le a_n x
		\bigr\}
	\Bigr) - \PP\bigl( M_{yR}(\dnm) \le a_n x \bigr) \\ &
	\le \PP \Bigl(
		\bigcup_{z\in P_n} \bigl\{
		M_{yR}(\Jz\setminus \Hz) > a_n x
		\bigr\}
	\Bigr)	\\ &
	\le 
	\sum_{z\in P_n}
	\PP \bigl( M_{yR}(\Jz \setminus \Hz) > a_n x \bigr).
\end{align*}
Secondly, using \eqref{eq:approxmixing1} above,
\begin{align*}
	\abs[\Big]{
	\PP \Bigl(
		\bigcap_{z\in P_n} \bigl\{
		M_{yR}(\Hz) \le a_n x
		\bigr\}
	\Bigr) - 
	\prod_{z\in P_n} 
	\PP \bigl( M_{yR}(\Hz) \le a_n x \bigr)}
	\to 0
\end{align*}
as $n\to\infty$. As
\begin{align*}
	0 &
	\le \prod_{z\in P_n} 
	\PP \bigl( M_{yR}(\Hz) \le a_n x \bigr)
	-
	\prod_{z\in P_n} 
	\PP \bigl( M_{yR}(\Jz) \le a_n x \bigr) \\ &
	\le 
	\sum_{z\in P_n}
	\PP \bigl( M_{yR}(\Jz \setminus \Hz) > a_n x \bigr) ,
\end{align*}
which tends to $0$ by Lemma~\ref{lem:tailapprox}, these considerations show that
\[
	\PP\bigl( M_{yR}(\dnm) \le a_n x \bigr) = 
	\prod_{z\in P_n} 
	\PP \bigl( M_{yR}(\Jz) \le a_n x \bigr) + o(1) .
\]
Note that this is also true with $P_n$ replaced by $Q_n$ and $\dnm$ replaced by $\dnp$, respectively. Consequently, as $\dnm \subseteq D_n \subseteq \dnp$, this implies that
\begin{align*}
	\prod_{z\in \Qn}\PP \bigl( M_{yR}(\Jz) \le a_n x \bigr) + o(1) &
	\le 
	\PP\bigl(\max_{v \in D_n} y_vR_v \le a_n x \bigr)  \\ &
	\le
	\prod_{z\in \Pn}\PP \bigl( M_{yR}(\Jz) \le a_n x \bigr) + o(1) 
\end{align*}
as $n\to\infty$. Utilizing a series representation of the logarithm in combination with Lemma~\ref{lem:expappr0}, we obtain
\begin{align*}
	\MoveEqLeft
	\prod_{z\in \Pn}\PP \bigl( M_{yR}(\Jz) \le a_n x \bigr) \\ &
	= \exp \Bigl(
		\sum_{z\in P_n} \log \bigl( 1-\PP(M_{y R}(\Jz) > a_n x)
		\bigr)
	\Bigr) \\ &
	= \exp \Bigl(-
		\sum_{z\in P_n} \PP(M_{yR}(\Jz) > a_n x)
		- \sum_{z\in\Pn} \sum_{j=2}^\infty \frac{1}{j} 
	\bigl(\PP(M_{yR}(\Jz) > a_n x)
	\bigr)^j
		\Bigr) \\ &
		= \exp \Bigl(-
		\sum_{z\in P_n} \PP(M_{yR}(\Jz) > a_n x) \Bigr) + o(1)
\end{align*}
as $n\to\infty$, and again similarly for $Q_n$. Due to \eqref{eq:sumapprox2}, this shows \eqref{eq:maxprobequality}.

The inequality \eqref{eq:maxprobinequality} is a consequence of \eqref{eq:maxprobequality}, Boole's inequality and \eqref{eq:sumapprox1}.
\end{proof}

\begin{lemma}\label{lem:JsetandAset}
Let $(R_v)$ be either $(Z_v)$ or $(Z_v^t)$ for some fixed $t\in\NN$. For almost all realizations $(y_v)$ of $(Y_v)$,
\begin{equation}\label{eq:maxformula1}
	\sum_{z\in\Pn} \PP \bigl( M_{yR}(\Jz) > a_n x \bigr)
	= \sum_{v\in D_n} \PP \bigl( M_{yR}(\Av) \le a_n x < y_v R_v \bigr) + o(1)
\end{equation}
as $n\to\infty$. In particular,
\begin{equation}\label{eq:maxformula2}
	\PP\bigl(\max_{v \in D_n} y_v R_v \le a_n x \bigr)  
	=
	\exp\Bigl(-\sum_{v\in D_n} \PP \bigl( M_{yR}(\Av) \le a_n x < y_v R_v \bigr) \Bigr) + o(1)
\end{equation}
as $n\to\infty$.
\end{lemma}

\begin{proof}
Clearly \eqref{eq:maxformula2} follows from \eqref{eq:maxformula1} due to Lemma~\ref{lem:maxvsJkasse}.

For all $z\in\Pn$ we find that
\begin{align*}
	\PP \bigl( M_{yR}(\Jz) > a_n x \bigr) &
	= \sum_{v \in \Jz} \PP \Bigl( M_{yR}\big(\{u \in \Jz:v \prec u \}\big) \le a_n x < y_v R_v \Bigr) \\ &
	\ge \sum_{v \in \Jz} \PP \bigl( M_{yR}(\Av) \le a_n x < y_v R_v \bigr) .
\end{align*}
Summing over $z\in P_n$ this implies that
\begin{align*}
	\sum_{z\in P_n} \PP \bigl( M_{yR}(\Jz) > a_n x \bigr) &
	\ge \sum_{v\in \dnm}\PP \bigl( M_{yR}(\Av) \le a_n x < y_vR_v \bigr) \\ & =
	\sum_{v\in D_n}\PP \bigl( M_{yR}(\Av) \le a_n x < y_vR_v \bigr) + o(1) ,
\end{align*}
where the equality is due to \eqref{eq:sumapprox1} since $\dnm \subseteq D_n \subseteq \dnp$.

To obtain an upper bound, 
we consider the set $ \overline B_z^{n}$ defined in \eqref{eq:overlineBz}. I.e. $ \overline B_z^{n}$ is the union of the $3^d-1$ neighboring $J^n$-boxes of $\Jz$.
Then
\begin{align*}
	\MoveEqLeft
	\PP \bigl( M_{yR}(\Jz) > a_n x,\, M_{yR}(\overline B_z^n) \le a_n x \bigr) \\ & 
	= \sum_{v \in \Jz} 
	\PP \Bigl( M_{yR}\big(\{u \in \Jz:v \prec u \}\big) \le a_n x,\, M_{yR}(\overline B_z^n) \le a_n x< y_v R_v \Bigr) \\ &
	\le \sum_{v \in \Jz} 
	\PP \bigl( M_{yR}(\Av) \le a_n x < y_v R_v \bigr) .
\end{align*}
Hence, summing over $z\in P_n$, we have the upper bound
\begin{align*}
	\MoveEqLeft
	\sum_{z\in P_n} \PP \bigl( M_{yR}(\Jz) > a_n x,\, M_{yR}(\overline B_z^n) \le a_n x \bigr)  \\ &
	\le \sum_{v \in D_n} 
	\PP \bigl( M_{yR}(\Av) \le a_n x < y_v R_v \bigr) + o(1)
\end{align*}
as $n\to\infty$. Consequently, the result follows if
\begin{equation*}%\label{eq:maxJogB1}
	\sum_{z\in\Pn}\PP \bigl( M_{yR}(\Jz) > a_n x,\, M_{yR}(\overline B_z^n) > a_n x \bigr)
	\to 0
\end{equation*}
as $n\to\infty$.

Let $E_{z,n}^+$ and $E_{z,n}^-$ be the discrete strip of size $\bb \gamma_n$ on the outside and inside of $\Jz$, respectively. E.g. 
\[
	E_{z,n}^+ = \{u\in\Zd\setminus \Jz
	\::\:
	z_\ell t_{n,\ell}-\gamma_{n,\ell}
	\le u_\ell
	\le (z_\ell +1) t_{n,\ell} - 1 + \gamma_{n,\ell},\ \text{for all }\ell
	\} .
\]

Using \eqref{eq:approxmixing2} in the second inequality below we find that
\begin{align*}
	\MoveEqLeft	\sum_{z\in P_n}\PP \bigl( M_{yR}(\Jz) > a_n x, M_{yR}(\overline B_z^n) > a_n x \bigr) \\ &
	\le \sum_{z\in P_n}\PP\bigl(M_{yR}(E_{z,n}^+) > a_n x, M_{yR}(E_{z,n}^-) > a_n x \bigr) \\ & \quad
	+ \sum_{z\in P_n}\PP \bigl( M_{yR} (\Jz) > a_n x,\, M_{yR}(\overline B_z^n\setminus E_{z,n}^+) > a_n x \bigr) \\ & \quad
	+ \sum_{z\in P_n}\PP \bigl( M_{yR}(\Jz\setminus E_{z,n}^-) > a_n x,\, M_{yR}(\overline B_z^n) > a_n x \bigr) \\ &
	\le \sum_{z\in P_n} \PP\bigl(M_{yR}(E_{z,n}^-) > a_n x \bigr) \\ & \quad
	+ 2 \sum_{z\in P_n} \PP \bigl( M_{yR}(\Jz) > a_n x \bigr)\,
	\PP \bigl( M_{yR}(\overline B_z^n) > a_n x \bigr) + o(1) .
\end{align*}
The first term tends to $0$ by Lemma~\ref{lem:tailapprox}. Since $\abs{\overline B_z^n}=(3^d-1)\abs{\Jz}$, the second term can also be seen to converge to $0$ using Lemma~\ref{lem:ergodictheorem} of Appendix~\ref{app:ergodictheorem} and arguments similar to those in the proof of Lemma~\ref{lem:tailapprox} thus concluding the proof.
\end{proof}

\end{appendix}

\bibliographystyle{imsart-number} % Style BST file (imsart-number.bst or imsart-nameyear.bst)

\end{document}